\documentclass{amsart}
\usepackage{amsmath,amsfonts,euscript,amssymb,amsthm,graphicx,subfigure,verbatim}
\usepackage{graphicx}
\usepackage{tikz}
\usepackage{hyperref}
\usepackage{mathtools,stmaryrd}
\usetikzlibrary{calc,through,arrows,plotmarks,shadows,shapes,shapes.multipart,positioning,fit,fadings,matrix}

\newtheorem{theorem}{Theorem}

\newtheorem{lemma}{Lemma}
\newtheorem{proposition}{Proposition}

\newcommand{\R}{\mathbb{R}}
\newcommand{\N}{\mathbb{N}}
\newcommand{\Z}{\mathbb{Z}}
\newcommand{\C}{\mathbb{C}}

\newcommand{\St}{\mathbb{S}^2}

\newcommand{\sgn}{\mathop{\mathrm{sgn}}}

\newcommand{\dd}[2]{\frac{\partial #1}{\partial #2}}

\newcommand{\m}{\boldsymbol{m}}

\newcommand{\x}{\boldsymbol{x}}

\newcommand{\e}{\boldsymbol{e}}
\newcommand{\eps}{\varepsilon}
\newcommand{\del}{\partial}
\newcommand{\ein}{\boldsymbol{\hat{e}}}

\newcommand{\bs}[1]{ \boldsymbol{#1}}

\newcommand{\ra}{\rangle}
\newcommand{\la}{\langle}
\def\Xint#1{\mathchoice
{\XXint\displaystyle\textstyle{#1}}%
{\XXint\textstyle\scriptstyle{#1}}%
{\XXint\scriptstyle\scriptscriptstyle{#1}}%
{\XXint\scriptscriptstyle\scriptscriptstyle{#1}}%
\!\int}
\def\XXint#1#2#3{{\setbox0=\hbox{$#1{#2#3}{\int}$}
\vcenter{\hbox{$#2#3$}}\kern-.5\wd0}}

\def\dashint{\Xint-}

\begin{document}

\title[Compactness results of chiral skyrmions]{Compactness results for static and dynamic chiral skyrmions near the conformal limit}%large-energy limit}

\author{Lukas D\"oring}
\address{RWTH Aachen\\Lehrstuhl I f\"ur Mathematik\\Pontdriesch 14-16\\52056 Aachen}
\address{JARA -- Fundamentals of Future Information Technology}
\email{l.doering@math1.rwth-aachen.de}

\author{Christof Melcher}
\address{RWTH Aachen\\Lehrstuhl I f\"ur Mathematik\\Pontdriesch 14-16\\52056 Aachen}
\address{JARA -- Fundamentals of Future Information Technology}
\email{melcher@rwth-aachen.de}

\begin{abstract}
We examine lower order perturbations of the harmonic map problem from $\R^2$ to $\mathbb{S}^2$ including chiral interaction in form of a helicity term that prefers modulation, and a potential term that enables decay to a uniform background state. Energy functionals of this type arise in the context of magnetic systems without inversion symmetry.  In the almost conformal regime, where these perturbations are weighted with a small parameter, we examine the existence of relative minimizers in a non-trivial homotopy class, so-called chiral skyrmions, strong compactness of almost minimizers, and their asymptotic limit. Finally we examine dynamic stability and compactness of almost minimizers in the context of the Landau-Lifshitz-Gilbert equation including spin-transfer torques arising from the interaction with an external current.  
\end{abstract}

\maketitle

%{\color{red}Journal: Calc. Var. $\leadsto$ F.H. Lin?}

\section{Introduction and main results}

Isolated chiral skyrmions are homotopically nontrivial field configurations $\m:\R^2 \to \St$ occurring as
relative energy minimizeres in magnetic systems without inversion symmetry. In such systems the 
leading-order interaction is Heisenberg exchange in terms of the Dirichlet energy
\[
D(\m) = \frac{1}{2} \int_{\R^2} |\nabla \m|^2 \, dx.
\]
Chiral interactions, in magnetism known as antisymmetric exchange or Dzyaloshinskii-Moriya interactions, are introduced in terms of Lifshitz invariants, the components of the tensor
$\nabla \m \times \m$.  A prototypical form is obtained by taking the trace, which yields the helicity functional
%\footnote{This formulation of the helicity functional agrees with the more standard expression $\int_{\R^2} \m \cdot \nabla \times \m \, dx$, provided $\m$ decays to $\ein_3$ %sufficiently rapidly.}% \LD{Do we consider $\m\in \mathcal{M}$ here?}
\[
H(\m) =  \int_{\R^2} \m  \cdot \left(\nabla \times \m \right) \,  dx,
\]
well-defined for moderately smooth $\m$ that decay appropriately to a uniform background state.
Extensions to the canonical energy space will be discussed later.

Chiral interactions are sensitive to independent rotations and reflections in the domain $\R^2$ and
the target $\St$, and therefore select specific field orientations. The helicity prefers curling configurations.
The uniform background state $\m(x) \to \ein_3$ as $|x| \to \infty$ is fixed by a potential energy $V(\m)= V_p(\m)$
depending on a power $2 \le p \le 4$ with 
\[
V_p(\m)= \frac{1}{2^{p}} \int_{\R^2} |\m -\ein_3|^{p} \, dx.
\]
The borderline case $p=2$ corresponds to the classical Zeeman interaction with an external magnetic field. 
The case $p=4$ turns out to play a
particular mathematical role in connection with helicity. From the point of view of physics, since $\tfrac{1}{4}|\m-\ein_3|^4=|\m-\ein_3|^2+(\m \cdot \ein_3)^2-1$, the
case $p=4$ features a specific combination of Zeeman and in-plane anisotropy interaction. Upon scaling, 
the governing energy functional
\[
E_\eps(\m) = D(\m) + \eps \bigl(H(\m)+V(\m) \bigr) 
\]
only depends on one coupling constant $\eps>0$. For $p=2$ variants of this functional have been 
examined in physics literature, see e.g. \cite{Bogdanov_Hubert:1994, Bogdanov_Hubert:1999, han2010}, predicting
the occurrence of specific topological defects, so-called chiral skyrmions, arranged in a regular lattice or as isolated topological soliton.  
In our scaling, tailored towards an asymptotic analysis, the parameter $\eps$ corresponds to the inverse
of the renormalized strength of the applied field. The almost conformal regime % when $\eps$ is sufficiently small
$0<\eps\ll 1$
features the ferromagnetic phase of positive energies, where $H$ is dominated by $D$ and $V$, i.e. $E_\eps(\m) \gtrsim D(\m)+\eps V(\m)$.
In this case the configuration space 
\[
\mathcal{M}=\{ \m:\R^2 \to \St: D(\m)+V(\m)< \infty\},
\]
admits the structure of a complete metric space (see below).
In the ferromagnetic regime, $\m \equiv \ein_3$ is the unique global energy minimizer,
while chiral skyrmions are expected to occur as relative energy minimizers in a nontrivial homotopy class.
In the case $p=2$ and for $0<\eps\ll 1$ this has been proven in \cite{CSK}.\\

Homotopy classes are characterized by the topological charge (Brouwer degree)
\[
Q(\m) = \frac{1}{4\pi} \int_{\R^2}   \m \cdot \left( \del_1 \m \times \del_2 \m \right) \, dx \in \Z, 
\]
which decomposes the configuration space into its path-connected components, the topological sectors.
In view of the background state $\ein_3$, the specific topological charge $Q(\m)=-1$ is energetically selected by the presence of a chiral interaction.
In fact, for all $2 \le p \le 4$ we have
%For all $0<\eps < 1/2$ and $\m \in \mathcal{M}$ of positive topological charge  we have $E_\eps(\m) >  4\pi \, Q(\m)$, while 
\[
\inf\left\{ E_\eps(\m): \m \in \mathcal{M} \; \text{with} \; Q(\m)=-1\right\}< 4\pi \quad \text{for} \quad  \eps>0,
\]
less than the classical topological lower bound for the Dirichlet energy, while
\[
\inf\left\{ E_\eps(\m): \m \in \mathcal{M} \; \text{with} \; Q(\m) \notin\{0,-1\}\right\}>4\pi
\quad \text{for} \quad  \eps \ll 1,
\]
a consequence of the energy bounds provided in Section \ref{sec:energy_bounds}.\\

These properties are in contrast to two-dimensional versions of the classical Skyrme functional (see e.g. \cite{Piette:1995, Arthur}) featuring full rotation and reflection 
symmetry. Here, the helicity term is replaced by the
the Skyrme term
\[
S(\bs u)= \frac{1}{4} \int_{\R^2} |\del_1 \bs u \times \del_2 \bs u|^2 \, dx, 
\]
a higher order perturbation of $D(\bs u)$, which prevents a finite energy collapse of the topological charge due to concentration effects. In particular,
the energy functional $D(\bs u)+ \lambda S(\bs u)+ \mu V(\bs u)$, for positive coupling constants $\lambda, \mu$, has an energy range above $4\pi$
in every non-trivial homotopy class. In the case $p=4$, the attainment of least energies for unit charge configurations and topologically non-trivial configurations
has been examined in \cite{Lin_Yang, Lin_Yang_splitting, 2D_skyrmion} and  \cite{Lin_Yang_splitting}, respectively. Explicit minimizers arise for $p=8$, see \cite{Piette:1995}. We shall recover this situation in the chiral case for $p=4$. \\

%
%
% (cf. Lemma~\ref{lem:lowerbound}, which yields $E_\eps(\m)\geq 4\pi$ if $Q(\m)=+1$).
%As in \cite{CSK}, we will employ the helical derivatives
%\begin{align}\label{eq:helical_derivatives}
%\mathcal{D}_i^\kappa \m := \partial_i \m - \kappa \ein_i \times \m, \quad i=1,2,
%\end{align}
%in the derivation of lower bounds to the energy.
%

Our first result confirms existence of (global) minimizers of $E_\eps$ in $\mathcal{M}$, subject to the constraint $Q=-1$, extending the result in \cite{CSK} for $p=2$
to the whole range $2\leq p \leq 4$ of exponents:

\begin{theorem}[Existence of minimizers]\label{thm:1}
Suppose $2\leq p \leq 4$ and $0<\eps \ll 1$. Then the infimum of $E_\eps$ in $\mathcal{M}$ subject to the constraint $Q=-1$
is attained by a continuous map $\m_\eps$ in this homotopy class such that 
\begin{align}\label{eq:energy_bounds}
4\pi(1-4\eps) \leq E_{\eps}(\m_\eps) \leq 4\pi \bigl(1-2(p-2) \eps\bigr).
\end{align}
For $p=2$ and $0<\eps \ll 1$, we have, more precisely,
\[ E_\eps(\m_\eps) \leq 4\pi\left(1- \left(4+o(1)\right)\tfrac{\eps}{\lvert \ln \eps \rvert}\right). \]
If $p=4$, minimizers are characterized by the equation
\begin{align}\label{eq:hel_el}
\mathcal{D}_1 \m + \m \times \mathcal{D}_2 \m = 0 \quad \text{where} \quad \mathcal{D}_i \m= \del_i \m - \tfrac 1 2 \ein_i \times \m.
\end{align}
\end{theorem}

For $2\leq p < 4$, Theorem~\ref{thm:1} is obtained by a concentration-compactness argument similar to \cite{CSK, Lin_Yang}: Provided ``vanishing'' holds, we prove that the helicity functional becomes negligible, so that the energy of a minimizing sequence approaches $4\pi$, which contradicts the upper bound coming from Lemma~\ref{lem:upperbound} below. If ``dichotomy'' holds, the cut-off result Lemma~\ref{lem:cutoff} (see Appendix) yields a comparison function with an energy well below the global minimium in its homotopy class. Hence, neither vanishing nor dichotomy appear.

The case $p=4$ is special in the sense that vanishing can no longer be ruled out within our approach. However, upper and lower energy bounds match, so that an explicit energy-minimizer in form of a specifically adapted stereographic map $\m_0$ is available. It follows that $\m_0$ belongs to the class 
\[
\mathcal{C}:=\{\m:\R^2 \to \St 
:D(\m)= 4 \pi, \, Q(\m)= -1, \, \m(\infty)=\ein_3  \}
\]
consisting of anti-conformal (harmonic) maps of minimal energy. 
Recall that harmonic maps on $\R^2$ with finite energy extend to
harmonic maps on $\St$ (cf. \cite{Sacks_Uhlenbeck}) with a well-defined limit as $x \to \infty$. 

Anti-conformal maps are characterized by
the equation $\del_1 \m - \m \times \del_2 \m=0$, a geometric version of the Cauchy-Riemann equation. Hence, identifying $\R^2 \simeq \C$,
the moduli space of $\mathcal{C}$ is $\C \setminus \{0\} \times \C$. More precisely, $\mathcal{C}$ agrees with the two-parameter family of
maps $\m_0(z)=\Phi(az+b)$ for $z \in \C$, where $(a,b) \in \C \setminus \{0\} \times \C$ and $\Phi \colon \R^2\simeq \C \to \St$ is a stereographic map of negative degree
with $\Phi(\infty)=\ein_3$, cf. \cite[Lemma~A.1]{Brezis_Coron_How_To_Bubbles}.
Note that $\mathcal{C} \cap \mathcal{M}$ is empty in the limit case $p=2$.

In the context of the energies $E_\eps$, the degeneracy of a map $\m_0 \in \mathcal{C}$ with respect to the complex scaling parameter $a$ is lifted if 
\begin{enumerate}
\item it safisfies the Bogomolny type equation \eqref{eq:hel_el}, i.e. is also an energy minimizer subject to $Q=-1$ for $p=4$ and $\eps>0$ arbitrary, or
\item it is obtained from a family of chiral skyrmions $\{\m_\eps\}_{\eps \ll 1}$, 
which we prove for $2<p<4$ and conjecture in the limit cases $p\in\{2,4\}$:
\end{enumerate}

%
%Provided $1<p\leq 2$ and up to a shift in $\R^2$, which cannot be avoided due to the translation invariance of the problem, we identify the asymptotic limit of a family $\{\m_\eps\}_\eps$ of minimizers as suitably rescaled stereographic map:
%\begin{theorem}\label{thm:2}
%Suppose $2<p < 4$ and $\{\m_\eps\}_{\eps \ll 1} \subset \mathcal{M}$ is a family such that  \LD{$\mathcal{O}(\eps)$ could also be negative, couldn't it?}
%\[  Q(\m_\eps) = -1 \quad \text{and} \quad 4\pi-E_\eps(\m_\eps) = \mathcal{O}(\eps). \]
%Then for the distance induced by the metric $d$
%\[ \dist(\m_\eps,\mathcal{C}) \to 0 \quad \text{for}\quad \eps\to 0. \]
%%and
%%\begin{align}\label{eq:expEeps}
%%E_\eps(\m_\eps) = D(\m_0) + \eps \bigl(H(\m_0)+V(\m_0)\bigr) + o(\eps).  
%%\end{align}
%%In fact, 
%Modulo translation the limit $\m_0 \in \mathcal{C}$ {\color{red}exists and} is uniquely determined by % , so that% $Q(\m_0)=-1$, $D(\m_0)=4\pi$ and
%\[ {\color{red} \lim_{\eps\to 0} \eps^{-1}(E_\eps(\m_\eps)-4\pi)=}H(\m_0)+V(\m_0)=\min_{\m\in \mathcal{C}} \bigl(H(\m) + V(\m)\bigr) = {\color{red}8\pi(2-p)}. \]
%\end{theorem}
\begin{theorem}[Compactness of almost minimizers]\label{thm:2}
Suppose $2<p < 4$ and $\{\m_\eps\}_{\eps \ll 1} \subset \mathcal{M}$ is a family such that
\[  Q(\m_\eps) = -1 \quad \text{and} \quad E_\eps(\m_\eps) \leq 4\pi - C_0\eps \]
for some constant $C_0>0$.
%\LD{Strong convergence of $m_{3,\eps}$ is not clear to me. }
Then, we have:
\begin{enumerate}
\item There exists $\m_0\in\mathcal{C}$ so that for $\eps\to 0$, up to translations and a subsequence,
\[ \nabla \m_\eps \to \nabla \m_0 \quad\text{strongly in }L^2(\R^2) \]
and
\[
\ein_3 \cdot (\m_\eps-\m_0) \rightharpoonup 0 \quad \text{weakly in }L^\frac{p}{2}(\R^2). 
\]
%\[ \quad m_\eps \rightharpoonup m_0 \quad \text{in } L^p(\R^2), \quad 1-m_{3,\eps} \rightharpoonup 1-m_{3,0} \quad \text{in }L^\frac{p}{2}(\R^2). \]
%for the distance induced by the metric $d$ \LD{Check convergence for whole sequence (should be ok, since for every subsequence, we can find another subseq. with limit $0$.)}
%\[ \dist(\m_\eps,\mathcal{C}) \to 0 \quad \text{for}\quad \eps\to 0. \]
%and
%\begin{align}\label{eq:expEeps}
%E_\eps(\m_\eps) = D(\m_0) + \eps \bigl(H(\m_0)+V(\m_0)\bigr) + o(\eps).  
%\end{align}
%In fact, 
\item If $\{\m_\eps\}_{\eps \ll 1}$ satisfies the more restrictive upper bound
\[ E_\eps(\m_\eps) \leq 4\pi + \eps 
%\underbrace{
\min_{\m\in\mathcal{C}} \bigl(H(\m)+V(\m)\bigr)
%}_{=-8\pi(p-2)} 
+ o(\eps) \quad \text{for }\eps\to 0,  \]
%e.g., if $\{\m_\eps\}_\eps$ is a family of minimizers of $E_\eps$,
then, modulo translations, the whole family converges to a unique limit $\m_0 \in \mathcal{C}$, which is determined by %is uniquely determined by % , so that% $Q(\m_0)=-1$, $D(\m_0)=4\pi$ and
\[ H(\m_0)+V(\m_0) = \min_{\m\in \mathcal{C}} \bigl(H(\m) + V(\m)\bigr) = -8\pi(p-2), \]
such that $\ein_3 \cdot (\m_\eps-\m_0) \to 0$ strongly in $L^\frac{p}{2}(\R^2)$. 
%\[
%\lim_{\eps \to 0} D(\m_\eps) = D(\m_0) \quad \text{and} \quad \lim_{\eps \to 0} V(\m_\eps) =V(\m_0).
%\]
Moreover,
\[ \lim_{\eps\to 0} \eps^{-1}(E_\eps(\m_\eps)-4\pi)=  \min_{\m\in \mathcal{C}} \bigl(H(\m) + V(\m)\bigr). \]
\end{enumerate}
\end{theorem}

In particular, Theorem~\ref{thm:2} applies to the family $\{\m_\eps\}_{\eps>0}$ of minimizers that has been constructed in Theorem~\ref{thm:1}.
Fixing the adapted stereographic map
\[
\Phi \colon \R^2 \to \St, \quad
\Phi(x)= \left( \frac{2 x^\perp}{1+|x|^2}, - \frac{1-|x|^2}{1+|x|^2} \right),
\]
so that $Q(\Phi)=-1$ and $\Phi(\infty)=\ein_3$,
we have
\begin{equation}\label{eq:stereo}
\m_0(x)=\Phi\left(\frac{x}{2(p-2)} \right) \quad \text{for }x\in\R^2.
\end{equation}
It remains an open question whether for positive $\eps$ the minimizers $\m_\eps$ of $E_\eps$ in the homotopy class $\{ Q=-1\}$ are actually unique (up to translations) and axially symmetric. As a first step and for $2<p<4$, Theorem~\ref{thm:2} implies that $\m_\eps$ is at least close in $\dot{H}^1$ and $L^p$ to the unique, axially symmetric vector field $\m_0$ given above.
 
Similar to the existence of minimizers of $E_\eps$, Theorem~\ref{thm:2} is proven by means of P.~L. Lions' concentration-compactness principle. However, since the minimal energy tends to $4\pi$ as $\eps \to 0$, the argument of Theorem~\ref{thm:1} needs to be modified in a suitable way. In fact, in order to rule out ``dichotomy'', we will use the boundedness of the lower-order correction $H+V$ to the Dirichlet energy $D$, which comes from the matching upper and lower a-priori bounds to the minimal energy and is preserved by the cut-off result Lemma~\ref{lem:cutoff}. As a consequence, we obtain a comparison vector-field of non-zero degree with Dirichlet energy strictly below $4\pi$, contradicting the classical topological lower bound $D(\m)\geq 4\pi \lvert Q(\m) \rvert$. ``Vanishing'', on the other hand, would imply that the helicity functional becomes negligible along a sequence of (almost-)minimizers, which is again ruled out by the a-priori bounds.

%{\color{red}Needs to be improved:}
%Both Theorem~\ref{thm:1} and Theorem~\ref{thm:2} are obtained by means of P.~L. Lions concentration-compactness principle\footnote{The case $p=2$ has already been dealt with in \cite{Melcher_CSK}.} In fact, they are a consequence of the same compactness result Proposition~\ref{prop:compactness} (see Section~\ref{sec:??} below). By a decomposition argument as in \cite{Melcher_CSK}, the undesirable situation of vanishing contradicts the uniform positivity of $-H$ along minimizing sequences. Dichotomy, on the other hand, is ruled out by a cut-off lemma similar to \cite{LinYang}. One obtains a contradiction even for $\eps \to 0$, since the cut-off lemma does not affect the bounds on $H$ and $V$ and would therefore yield a map $\m$ of non-zero degree and Dirichlet energy less than $4\pi$.

\medskip

The second part of this paper addresses the dynamic stability of spin-current driven chiral skyrmions in the almost conformal regime $\eps \ll 1$. This is ultimately a question of regularity 
for the Landau-Lifshitz-Gilbert equation, for which finite time blow-up, typically accompanied by topological changes, has to be expected if energy accumulates to the critical threshold of $4\pi$.
In the presence of an in-plane spin-velocity $v \in \R^2$ the Landau-Lifshitz-Gilbert equation is given by
\begin{equation}\label{eq:LLG_STT}
\del_t \m + ( v \cdot \nabla) \m= \m \times \Big[ \alpha \, \del_t \m +
 \beta (v \cdot \nabla) \m - \bs h_\eps(\m) \Big]
\end{equation}
where $\alpha$ and $\beta$ are positive constants and
\[
\bs h_\eps(\m) = - \mathrm{grad} \, E_\eps(\m)
\]
is the effective field, see \cite{Cros, Schutte2014, Komineas2015} and \cite{Cote_Ignat_Miot, KMM_spin, Melcher_Ptashnyk} for a mathematical account.
In the Galilean invariant case $\alpha=\beta$ traveling wave solutions are obtained 
by transporting equilibria $\m \times \bs h_{\rm eff}=0$ along $c=v$. In the conformal case $\eps=0$, as observed in \cite{Komineas2015}, traveling wave solutions 
are obtained for arbitrary $\alpha$ and $\beta$ by transporting conformal or anti-conformal equilibria of unit degree along $c \in \R^2$ determined by the free Thiele equation 
\[
(c-v)^\perp = \alpha c - \beta v.
\]
%
%
%We shall prove that there exists an $\eps$-independent life span $T>0$ for regular solutions $\m_\eps=\m_\eps(x,t)$
%where $\m_\eps(x)=\m_\eps(x,0)$ with $x \in \R^2$ is an isolated chiral skyrmion with anti-conformal limit $\m_0$.
%Moreover, as $\eps \to 0$ we have 
%\[
%\m_\eps(\cdot , t) \to \m_0( \cdot- ct) \quad \text{strongly in energy for every $0 \le t \le T$.}
%\]
%In the case $\eps >0$ one expects an asymptotic propagation speed $c_\eps \approx c_0$
%\[
%4\pi (c_\eps-v)^\perp = D_\eps (\alpha c_\eps - \beta v) 
%\quad
%\text{where} 
%\quad D_\eps =  \lim_{t \to \infty} D(\m_\eps(\cdot, t))
%\] 
%with $D_\eps \approx 4 \pi$ for $\eps \ll 1$. 
%
%
%In the case of magnetic Ginzburg-Landau vortices, the corresponding Thiele equation including vortex and boundary interaction has been derived in \cite{}. \\

We are interested in the regime $0<\eps \ll 1$ for that case $p=4$.
Taking into account the asymptotic behavior of almost minimizers, it is natural to
pass to the moving frame 
\begin{equation}\label{eq:moving_frame}
\m (x,t) \mapsto \m (x+ c t,t) \quad \text{where} \quad  (c-v)^\perp = \alpha c - \beta v.
\end{equation}
After a rigid rotation in space (see Appendix \ref{ap:LLG}), this yields the pulled back equation
\begin{equation}\label{eq:LLG_MF}
(\del_t - \nu \del_z) \m = \m \times \Big[  \alpha  (\del_t-\nu \del_z) \m -  \bs h_\eps(\m)  \Big]
\end{equation}
with effective coupling parameter 
\[
\nu= \frac{2 (\alpha-\beta) v}{1+\alpha^2},
\] 
where $v>0$ is now the intensity of the spin current, and with the Cauchy-Riemann operator
\[
\del_z \m = \frac{1}{2} \left( \del_1 \m -\m \times \del_2 \m \right).
\]
revealing the conformal character of \eqref{eq:LLG_STT}.\\

Observe that any $\m \in \mathcal{C}$, which is also an equilibrium for the energy,
is a static solution for the pulled back dynamic equation, i.e. a traveling wave profile for \eqref{eq:LLG_STT}. For $\eps=0$, the pure Heisenberg model, every $\m \in \mathcal{C}$
is a minimizer, hence an equilibrium, recovering the observation from \cite{Komineas2015}. For $p=4$ and $\eps>0$ the matching upper energy bound characterizes 
$\m(x)=\Phi(x/4)$ with $\Phi$ given by \eqref{eq:stereo} not only as explicit energy minimizer within the class $\{Q=-1\}$  but also as an explicit static solution of \eqref{eq:LLG_MF},
i.e. an explicit traveling wave profile of \eqref{eq:LLG_STT}.

%For the pulled back equation we prove the following:% \LD{Add static results for $p=4$.}
\begin{theorem}[Existence, stability, compactness] \label{thm:3}  Suppose $p = 4$ and $0<\eps \ll 1$.
\begin{enumerate}
\item
There exists $\m \in \mathcal{C}$ independent of $\eps$, which minimizes the energy in its homotopy 
class and is a static solution of \eqref{eq:LLG_MF} and therefore a traveling wave profile for \eqref{eq:LLG_STT}.
\item Suppose $\{\m_{\eps}^{0}\}_{\eps \ll 1} \subset \mathcal{M}$ is a family of initial data with
$\nabla \m_{\eps}^{0} \in H^2(\R^2)$ and such that for a constant $c>0$ independent of $\eps$
\[
Q(\m_\eps^0)=-1 \quad \text{and} \quad E_\eps(\m_\eps^0) \le  4 \pi - c \eps.
\]
Then there exists a unique family $\{ \m_\eps\}_{\eps \ll 1} \subset C^0([0;T]; \mathcal{M})$ of local smooth solutions of \eqref{eq:LLG_MF} 
with initial data $\m_\eps(t=0)=\m_{\eps}^{0}$ for every
\[
0<T < \frac{c \alpha}{32 \pi (1+\alpha^2) \nu^2}.
\] 
\item If $\nabla \m^0_\eps \to \nabla \m_0$ strongly in $L^2(\R^2)$ for some $\m_0 \in \mathcal{M}$ as $\eps \to 0$,
 then $\m_0 \in \mathcal{C}$ and $\nabla \m_\eps(t) \to \nabla \m_0$ in $L^2(\R^2)$ for every $t \in [0,T]$.
\end{enumerate}
\end{theorem}

\subsection*{Outline of the paper}
The remainder of the paper is structured as follows: First, in Section~\ref{sec:energy_bounds}, we prove the upper and lower bounds \eqref{eq:energy_bounds} to the minimal energy $E_\eps$ in the homotopy class $\{Q=-1\}$, i.e. Lemmas~\ref{lem:lowerbound}~and~\ref{lem:upperbound}. In particular, we obtain the equation \eqref{eq:hel_el} characterizing minimizers in the case $p=4$.

In Section~\ref{sec:compactness}, we exploit the energy bounds and derive the first two main results, i.e. Theorems~\ref{thm:1}~and~\ref{thm:2}. In fact, both will be rather straightforward corollaries of a separate concentration-compactness result in the spirit of \cite{CSK}, i.e. Proposition~\ref{prop:compactness}.

Section~\ref{sec:dynamics} contains the proof of Theorem~\ref{thm:3}. The main point are regularity arguments in the spirit of \cite{Struwe:85}, which exploit
the energy bounds to rule out blow-up on a uniform time interval.

Finally, in the Appendix, we provide a few supplementary, technical results: A cut-off lemma similar to the ones used for example in \cite{CSK,Lin_Yang}, which enters the proof of Proposition~\ref{prop:compactness}; the explicit construction of a ``stream function'' that is needed in the upper-bound construction in Lemma~\ref{lem:upperbound} for $p=2$; and the derivation of \eqref{eq:LLG_MF}.
 
 \subsection*{Notation and preliminaries}
 Throughout the paper, we shall use the convention
\[ \nabla \times \m = \left(\begin{smallmatrix} \nabla \times m_3\\ \nabla \times m\end{smallmatrix}\right) \quad \text{for}\quad \m=\left(\begin{smallmatrix} m\\m_3 \end{smallmatrix}\right), \]
where
\[ \nabla \times m = \partial_1m_2-\partial_2m_1 \quad \text{and} \quad \nabla \times m_3 =%
 -\nabla^\perp m_3 =
  \left(\begin{smallmatrix} \partial_2 m_3\\-\partial_1 m_3 \end{smallmatrix}\right).  \]
  
We equip the space $\mathcal{M}=\{ \m:\R^2 \to \St \in H^1_{\rm loc}(\R^2)\; \text{with} \; D(\m)+V(\m)< \infty\}$ with the metric $d$ given as
\[
d(\m,\bs n)=\|\nabla(\m -\bs n)\|_{L^2}+\|\ein_3 \cdot( \m- \bs n)\|_{L^{\frac p 2}}.
\]
Completeness with respect to this metric follows from the fact that by virtue of the geometric constraint $\lvert \m \rvert^2=1$ 
we have $1-m_3=\frac 1 2 \lvert \m-\ein_3 \rvert^2$, so that
\[
V(\m) = 2^{-\frac{p}{2}}\int_{\R^2} (1-m_3)^{\frac{p}{2}} \, dx.
\]
 Depending on the context, it is convenient to use this alternative representation.
In order to extend the helicity to the configuration space $\mathcal{M}$ we recall that according to a variant (see e.g. \cite{Brezis_Coron}) 
of the approximation result by Schoen and Uhlenbeck \cite{Schoen_Uhlenbeck} 
 \[
 \mathcal{M}_0=\{ \m: \R^2 \to \St : \m -\ein_3 \in C^\infty_0(\R^2;\R^3)\}
 \]
 is a dense subclass of $\mathcal{M}$ with respect to the metric $d$. The compact support property can be achieved by
 a suitable cut-off as in Lemma \ref{lem:cutoff}.
 We have for $\m \in \mathcal{M}_0$
 \[
H(\m)= \int_{\R^2} (\m-\ein_3)\cdot \nabla \times \m \, dx
 \]
 while
 \[
  (\m-\ein_3)\cdot \nabla \times \m =  m \cdot \nabla \times m_3 - (1-m_3) \nabla \times m. 
 \]
Integration by parts shows that 
 \[
  \int_{\R^2} m \cdot \nabla \times m_3 \, dx = -\int_{\R^2} (1-m_3) \nabla \times m \, dx .
\]

The integral on the right extends uniquely to $\mathcal{M}$ since $L^{\frac p 2}$-convergence implies $L^2$-convergence for sequences of 
uniformly bounded functions. The integrand on the left is bounded by $(1-m_3^2)|\nabla \m|$, hence summable for $\m \in \mathcal{M}$ and the
integration by parts formula above holds true.  
Accordingly the energy $E_\eps(\m)=D(\m)+\eps(H(\m)+V(\m))$, initially defined on $\mathcal{M}_0$, 
extends to a continuous integral functional on $\mathcal{M}$
 \[
 E_\eps(\m)= \int_{\R^2} e_\eps(\m) \, dx
\]
with integrable density
\begin{equation} \label{eq:density}
e_\eps(\m)  = \tfrac{1}{2} \lvert \nabla \m \rvert^2 + \eps \Bigl( (\m-\ein_3) \cdot \nabla \times \m + \tfrac{1}{2^p} \lvert \m - \ein_3 \rvert^p \Bigr).
\end{equation}
For later purpose it will be convenient to introduce the topological charge density 
\[
\omega(\m) = \m \cdot \left( \del_1 \m \times \del_2 \m \right)
\]
entering the definition of topological charge
\[
Q(\m) = \frac{1}{4\pi} \int_{\R^2} \omega(\m) \, dx \in \Z
\]
for $\m \in \mathcal{M}_0$, which uniquely extends to $\mathcal{M}$ by virtue of Wente's inequality \cite{Wente, Helein_book}, and satisfies the classical topological lower bound
$D(\m) \ge 4 \pi |Q(\m)|$ for all $\m \in \mathcal{M}$.
 
\section{Energy bounds}\label{sec:energy_bounds}
Both the treatments of the static and dynamic problem rely on good upper and lower bounds to the energy $E_\eps$ in terms of $0<\eps\ll 1$. In fact, a major problem in extending our analysis to the physically relevant case $p=2$ consists in the lack of a lower bound that matches the logarithmic upper bound in Theorem~\ref{thm:1}. Due to the quadratic decay of the stereographic map $\Phi$ for $\lvert x \rvert \gg 1$, which leads to a logarithmically growing potential energy $V$ if $p=2$, we conjecture the logarithmic upper bound to be optimal in terms of scaling.

%\subsection*{Estimates for the helicity and lower energy bound}
%Suppose $\m - \ein_3 \in C^\infty_0(\R^2;\R^3)$.
From the above representations of $H$ and $V$ it follows %for $1 \le p \le 2$ that 
\begin{equation} \label{eq:helicity_bound}
\bigl( H(\m) \bigr)^2 \le 32 D(\m)V(\m) \qquad \forall \m\in \mathcal{M}.
\end{equation}
By Young's inequality we immediately infer the following lower energy bound:

\begin{lemma}[Boundedness in $\mathcal{M}$]\label{lem:bound_norm}
Suppose $2 \le p \le 4$ and $\eps>0$. Then,
\[
%E_\eps(\m) \ge \left(1 - 8 \eps \right)D(\m) \quad \text{and} \quad 
E_\eps(\m) \geq (1-16\eps) D(\m) + \tfrac{\eps}{2} V(\m) \qquad \text{for any }\m\in\mathcal{M}.
\]
%In particular, $E_\eps(\m) \geq 4\pi (1-16\eps)$ whenever $Q(\m)=\pm 1$.
\end{lemma}

Using the helical derivatives \eqref{eq:hel_el_kappa}, we can further improve the lower bound:
\begin{lemma}[Lower bound]\label{lem:lowerbound}
Suppose $2 \le p \le 4$, $\eps>0$ and $\m\in\mathcal{M} \setminus \{\m \equiv \ein_3\}$. Then
\[
%E_\eps(\m) \ge \left(1 - 8 \eps \right)D(\m) \quad \text{and} \quad 
E_\eps(\m) \geq 4\pi Q(\m)+\eps \Bigl(1-2\eps \, \tfrac{V_4(\m)}{V_p(\m)}\Bigr) V_p(\m) 
\]
and 
\[
%E_\eps(\m) \ge \left(1 - 8 \eps \right)D(\m) \quad \text{and} \quad 
E_\eps(\m) \geq \Bigl(1-2\eps \, \tfrac{V_4(\m)}{V_p(\m)}\Bigr) D(\m) + 8\pi \, \eps \, \tfrac{V_4(\m)}{V_p(\m)} \, Q(\m). 
\]
The second lower bound is attained if and only if 
\begin{align}\label{eq:hel_el_kappa}
\mathcal{D}_1^\kappa \m + \m \times \mathcal{D}_2^\kappa \m = 0, \quad \text{where} \quad \mathcal{D}_i^\kappa \m= \del_i \m - \kappa\,  \ein_i \times \m,
\end{align}
holds for $\kappa=\tfrac{V_4(\m)}{2V_p(\m)}$. In particular,  for $Q(\m)=-1$
\[ E_\eps(\m) \geq D(\m) \left(1-4\eps \tfrac{V_4(\m)}{V_p(\m)}\right) \geq 4\pi (1-4\eps). \]
\end{lemma}

A %matching
corresponding upper bound in the homotopy class $Q(\m)=-1$ is obtained by %means of an appropriately rescaled
rescaling the stereographic map $\Phi$ appropriately.
%\[
%\Phi \colon \R^2\to\St, \quad \Phi(x)= \left( \frac{2 x^\perp}{1+|x|^2}, - \frac{1-|x|^2}{1+|x|^2} \right)^T, \quad x\in\R^2.
%\]
For $p=2$, an additional cut-off procedure is needed.

\begin{lemma}[Upper bound]\label{lem:upperbound}
Suppose $2 \leq p \le 4$ and $\eps>0$. Then, there exists a smooth representative $\tilde{\m}\in\mathcal{M}$ in the homotopy class $Q=-1$ such that
\begin{align*}
\MoveEqLeft \inf \left\{ E_\eps(\m) : \m\in\mathcal{M},\; Q(\m)=-1 \right\}
\le  E_\eps(\tilde{\m})\\
&
\begin{cases}
=4 \pi \bigl(1 -2(p-2) \, \eps \bigr), &\text{if }2<p\leq 4,\\
\leq 4 \pi \left(1 - \bigl( 4+ o(1) \bigr)\tfrac{\eps}{\lvert \ln \eps \rvert} \right), &\text{if $p=2$ and $0<\eps\ll 1$}.
\end{cases}
\end{align*}
\end{lemma}
%Note that in the case $p=4$ upper and lower bound match, so that the vector field $\tilde{\m}$ in fact is a minimizer of $E_\eps$ in the homotopy class $Q=-1$.
For $p=4$, upper and lower bounds match, so that the vector field $\tilde{\m}$ actually is a minimizer of $E_\eps$ in the homotopy class $Q=-1$.% (see Lemma~\ref{lem:p4critpt} in the Appendix).

\begin{proof}[Proof of Lemma~\ref{lem:lowerbound}]
As in \cite{CSK} we will employ the helical derivatives $\mathcal{D}_i^\kappa$ as given in \eqref{eq:hel_el_kappa}
%\begin{align}\label{eq:helical_derivatives}
%\mathcal{D}_i^\kappa \m := \partial_i \m - \kappa \ein_i \times \m, \quad i=1,2,
%\end{align}
and appeal to the following relation from \cite[Proof of Lemma~3.2]{CSK}:

\medskip
\noindent \textbf{Step~1}: \textit{For any $\m\in\mathcal{M}$, we have
\begin{align*}
\MoveEqLeft
\tfrac{1}{2} \lvert \nabla \m \rvert^2 - \omega(\m) + \kappa \bigl( (\m-\ein_3) \cdot \nabla \times \m + 2\kappa \tfrac{1}{4} (1-m_3)^2 \bigr)\\
&= \lvert \mathcal{D}_1^\kappa \m + \m\times \mathcal{D}_2^\kappa \m  \rvert^2 \geq 0.
\end{align*}}

Indeed, using $\lvert \mathcal{D}_1^\kappa \m + \m \times \mathcal{D}_2^\kappa \m \rvert^2 = \lvert\mathcal{D}_1^\kappa \m \rvert^2 + \lvert \mathcal{D}_2^\kappa \m \rvert^2 + 2 \mathcal{D}_1^\kappa \m \cdot (\m \times \mathcal{D}_2^\kappa \m)$, the claim immediately follows from
\[ \lvert \mathcal{D}_1^\kappa \m \rvert^2 + \lvert \m\times \mathcal{D}_2^\kappa \m \rvert^2 = \lvert \nabla \m \rvert^2 + \kappa^2(1+m_3^2) + 2\kappa \, \m\cdot \nabla \times \m  \]
and
\[ \mathcal{D}_1^\kappa \m \cdot (\m\times \mathcal{D}_2^\kappa \m) = -\omega(\m) -\kappa^2 m_3 - \kappa \,\ein_3 \cdot \nabla \times \m.  \]

\medskip
\noindent \textbf{Step~2}: \textit{Conclusion.}
Recall that for $2\leq p \leq 4$
\[ V(\m) = V_p(\m) = \int_{\R^2} \bigl(\tfrac{1}{2}(1-m_3) \bigr)^{\frac{p}{2}} \,dx. \]
%\LD{Might prefer not to introduce $V_p$ and directly get the lower bound.}
Choosing $\kappa=\eps$ in Step~1 and integrating over $\R^2$, the first claim follows as in \cite{CSK}.\\

With the choice of  $\kappa=\tfrac{V_p(\m)}{2V_4(\m)}$ it follows that
\[ D(\m)- 4\pi Q(\m) + \tfrac{V_p(\m)}{2V_4(\m)} \bigl( H(\m) + V_p(\m) \bigr) \geq 0, \]
i.e.
\[ H(\m) + V(\m) \geq -2\tfrac{V_4(\m)}{V_p(\m)}\bigl( D(\m)- 4\pi Q(\m) \bigr). \]
Hence, we obtain the second lower bound:
\begin{align*}
\MoveEqLeft E_\eps(\m) = D(\m) + \eps \bigl( H(\m)+V(\m) \bigr) \geq D(\m) - 2\eps \tfrac{V_4(\m)}{V_p(\m)}\bigl( D(\m) - 4\pi Q(\m) \bigr).
%&\geq (1-2\eps) D(\m) + 8\pi \eps Q(\m).
\end{align*}
In particular, Step~1 implies that the inequality is sharp if and only if \eqref{eq:hel_el_kappa} holds for $\kappa=\tfrac{V_p(\m)}{2V_4(\m)}$. \\

If $Q(\m)=-1$, we can use the classical topological lower bound $D(\m)\geq 4\pi \lvert Q(\m) \rvert = 4\pi$ to conclude
\begin{align*}
\MoveEqLeft E_\eps(\m) \geq D(\m) - 2\eps \tfrac{V_4(\m)}{V_p(\m)} \bigl( D(\m) - 4\pi Q(\m) \bigr)\\
&\geq \Bigl( 1 - 4\eps \underbrace{\tfrac{V_4(\m)}{V_p(\m)}}_{\leq 1}\Bigr) D(\m) \geq 4\pi(1-4\eps). \qedhere
\end{align*}
\end{proof}

%\subsection*{Skyrmion number and upper energy bounds}

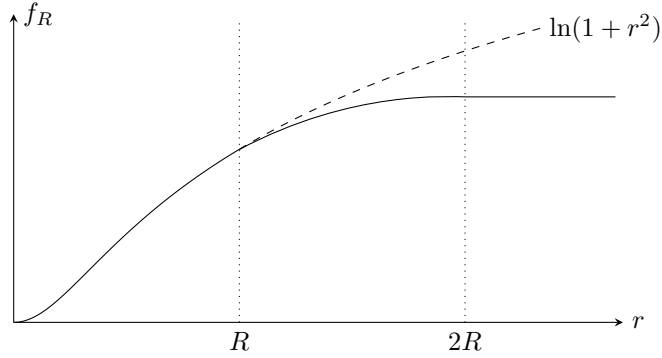
\begin{figure}
\begin{tikzpicture}
\draw [->,>=stealth] (0,0) -- (8.1,0) node [right] {$r$};
\draw [->,>=stealth] (0,0) -- (0,4.1) node [right] {$f_R$};
\node at (3,0) [below] {$R$};
\node at (6,0) [below] {$2R$};
\draw [dotted] (3,0) -- (3,4);
\draw [dotted] (6,0) -- (6,4);
\draw [variable=\x,domain=0:3,samples=100] plot ({\x},{ln(1+\x*\x)});
\draw [variable=\x,domain=3:7,dashed,samples=100] plot ({\x},{ln(1+\x*\x)}) node [right] {$\ln(1+r^2)$};
\draw (3,2.3) .. controls (4.5,3.1) and (5.8,3) .. (6,3);
\draw (6,3) -- (8,3);
\end{tikzpicture}
\caption{A sketch of the ``stream function'' $f_R$ that the upper-bound construction for $p=2$ is based on.}
\label{fig:streamfct}
\end{figure}

\begin{proof}[Proof of Lemma~\ref{lem:upperbound}]
If $2<p\leq 4$, we may just define
\[ \tilde{\m}\colon \R^2 \to \St, \quad \tilde{\m}(x):=\Phi_\lambda(x):=\Phi(\lambda x), \]
for $\lambda>0$ yet to be determined. Since $D(\Phi)=4\pi$, $H(\Phi)=-8 \pi$, $V(\Phi)= 2\pi /(p-2)$ and
\[
E_\eps(\Phi_{\lambda^*}) = \min_{\lambda >0} E_\eps(\Phi_\lambda)=  D(\Phi) - \frac{\eps H(\Phi)^2}{4 V(\Phi)}, \quad \lambda^*=-\frac{2V(\Phi)}{H(\Phi)},
\]
by a simple scaling argument, we obtain the claim with $\lambda=\lambda^*=(2(p-2))^{-1}$.

For $p=2$, however, $\Phi\not\in\mathcal{M}$, since the potential energy $V(\Phi)$ diverges logarithmically. Thus, $\Phi$ needs to be cut off in a suitable way. To this end, for $R \gg 1$ to be chosen later, we fix a smooth function $f_R \colon [0,\infty) \to \R$ (see Figure~\ref{fig:streamfct} and the Appendix for an explicit construction) so that
\[
f_R(r)=\begin{cases}\ln(1+r^2), &\text{for }0\leq r \leq R,\\ \text{const.}, &\text{for }r\geq 2R,\end{cases}
\]
and, denoting by $0<C<\infty$ a generic, universal constant, whose value may change from line to line:
\[ 
0\leq f'_R(r)\leq \tfrac{2r}{1+r^2}, \quad
0\leq -f_R''(r) \leq \tfrac{C}{1+r^2}, \quad   \text{for all }r\geq R.
\]
Then, we define a smooth vector field $\Phi_R\colon \R^2\to\St$ via
\[ 
\Phi_R(x):= \left( f_R'(\lvert x \rvert) \frac{x^\perp}{\lvert x \rvert}, \sgn\bigl(\lvert x \rvert-1\bigr) \sqrt{1-\bigl(f'_R(\lvert x \rvert)\bigr)^2}  \right)^T,  \quad  x\in\R^2.
\]
Note that $\Phi_R=\Phi$ on $B_R$ and $\Phi_R=\ein_3$ on $\R^2\setminus B_{2R}$. On $A_R:=B_{2R}\setminus B_R$, we have
\[ \lvert \nabla \Phi_R(x) \rvert^2 \leq \tfrac{C}{\lvert x \rvert^4}, \qquad \lvert \Phi_R(x) - \ein_3 \rvert^2 \leq \tfrac{C}{\lvert x \rvert^2}, \quad x\in A_R. \]
Hence, we compute in polar coordinates
\begin{align*}
\int_{B_R} \tfrac{1}{2}\lvert \nabla \Phi_R \rvert^2 dx &= 4\pi \int_0^R \tfrac{2r}{(1+r^2)^2} \, dr \leq 4\pi,\\
\int_{B_R} \tfrac{1}{4} \lvert \Phi_R - \ein_3 \rvert^2 dx &= \pi \int_0^R \tfrac{2r}{1+r^2} \, dr = \pi \ln(1+R^2),\\
\int_{A_R} \tfrac{1}{2} \lvert \nabla \Phi_R \rvert^2 dx &\leq C \int_R^{2R} \tfrac{1}{r^3} \, dr = \tfrac{C}{R^2},\\
\int_{A_R} \tfrac{1}{4} \lvert \Phi_R - \ein_3 \rvert^2 dx &\leq C \int_R^{2R} \tfrac{1}{r} \, dr = C.
\end{align*}
The region $\R^2\setminus B_{2R}$ does not contribute to the energy. In particular, we have
\[ \lvert Q(\Phi)-Q(\Phi_R) \rvert \leq C\int_{B_R^C} \lvert \nabla \Phi \rvert^2 + \lvert \nabla \Phi_R \rvert^2 \, dx \ll 1 \quad \text{provided }R\gg 1, \]
so that $Q(\Phi_R)=Q(\Phi_R)=-1$.

In order to estimate the contribution from the helicity, we exploit that
\begin{align*}
\MoveEqLeft \left(\begin{smallmatrix} \Phi_{1,R}\\\Phi_{2,R}\end{smallmatrix}\right)(x) \cdot \nabla \times \Phi_{3,R}(x) 
= \sgn\bigl(\lvert x \rvert-1\bigr) \frac{\bigl(f'_R(\lvert x \rvert)\bigr)^2 f''_R(\lvert x \rvert)}{\sqrt{1-\bigl( f'_R(\lvert x \rvert) \bigr)^2}}\\
&\begin{cases}
=-8\tfrac{\lvert x \rvert^2}{(1+\lvert x \rvert^2)^3}, &\text{for }0\leq \lvert x \rvert \leq R,\\
\leq 0, & \text{for } \lvert x \rvert \geq R.
\end{cases}
\end{align*}
Hence, using $\tfrac{d}{dr} \tfrac{r^4}{(1+r^2)^2}=4 \tfrac{r^3}{(1+r^2)^3}$, we find
\[ H(\Phi_R)= 2\int_{\R^2} \left(\begin{smallmatrix} \Phi_{1,R}\\\Phi_{2,R}\end{smallmatrix}\right) \cdot \nabla \times \Phi_{3,R} \, dx \leq -32\pi \int_{0}^R \bigl(\tfrac{r}{1+r^2}\bigr)^3 \, dr = -8\pi \tfrac{R^4}{(1+R^2)^2}.  \]
Summarizing, for sufficiently large $R\gg 1$, we have obtained
\begin{align*}
D(\Phi_R) &\leq 4\pi + \tfrac{C}{R^2},\\
H(\Phi_R) &\leq -8\pi \bigl( \tfrac{R^2}{R^2+1}\bigr)^2 \leq -8\pi + \tfrac{C}{R^2},\\
V(\Phi_R) &\leq \pi\ln(1+R^2) + C.
\end{align*}
Defining
\[ \tilde{\m} \colon \R^2 \to \St, \quad \tilde{\m}(x) = \Phi_R(\lambda x), \]
where $\lambda>0$ will be chosen below, and rescaling, we arrive at
\begin{align*}
\MoveEqLeft E_\eps(\tilde{\m}) = D(\Phi_R) + \eps \lambda^{-1} \bigl(H(\Phi_R)+\lambda^{-1} V(\Phi_R) \bigr)\\
&\leq 4\pi + \tfrac{C}{R^2} + \eps\lambda^{-1} \bigl( -8\pi + \lambda^{-1} \pi \ln(1+R^2) + C(R^{-2}+\lambda^{-1}) \bigr).
\end{align*}
Now, choose $R = \eps^{-\frac{1}{2}} \lvert \ln \eps \rvert$ and let $\lambda=L\lvert \ln \eps \rvert$ for $L>0$ fixed and $0<\eps\ll 1$. Then,
\[ E_\eps(\tilde{\m}) \leq 4\pi + \tfrac{\eps}{\lvert \ln \eps \rvert} \bigl( -\tfrac{8\pi}{L} + \tfrac{\pi}{L^2} + o(1)\bigr) \quad \text{for }0<\eps \ll 1, \]
which turns into the claim for $L=\tfrac{1}{4}$.
%if $M$ is chosen sufficiently large but fixed,% for $M$\ that $C/R^2 \leq \pi $, we arrive at
%\begin{align*}
%\MoveEqLeft E_\eps(\tilde{\m}) \leq 4\pi - \tfrac{\eps}{\ln\frac{1}{\eps}} \left(7\pi - \tfrac{C}{M} - \pi \tfrac{\ln(1+M\eps^{-2}\ln^2(\frac{1}{\eps}))}{\ln\frac{1}{\eps}} - \tfrac{C}{\ln\frac{1}{\eps}} \right)\\
%&\leq 4\pi(1-\eps/\ln\tfrac{1}{\eps}) \qquad \text{for }0<\eps\ll 1. \qedhere
%\end{align*}
\end{proof}

\section{Compactness and proofs of Theorems~\ref{thm:1}~and~\ref{thm:2}}\label{sec:compactness}%expansion of $E_\eps$ for $0< \eps \ll 1$}
In this section, we prove existence of minimizers $\m_\eps$ of $E_\eps$ under the constraint $Q=-1$, and their strong convergence to a unique harmonic map $\m_0\in\mathcal{C}$ as $\eps\to 0$. In fact, both results rely on P. L. Lions' concentration-compactness principle. We state the common part as a separate compactness result -- Proposition~\ref{prop:compactness} -- from which Theorems~\ref{thm:1}~and~\ref{thm:2} can be deduced easily:
%By means of P. L. Lions concentration-compactness principle, we may prove simultaneously existence of minimizers of $E_\eps$ in $\mathcal{M}$ for fixed $\eps>0$ subject to the constraint $Q=-1$, and compactness as $\eps\to 0$ of sequences of magnetization configurations $\{\m_\eps\}_{\eps}\subset \mathcal{M}$ with $Q(\m_\eps)=-1$ and non-negative ``excess energy'' $4\pi-E_\eps(\m_\eps)$ of the order $\eps$.
%
%Both results are in fact a consequence of the following Proposition: %, we will prove both compactness results simultaneously:% of minimizing sequences for $E_\eps$, $\eps>0$ fixed, and sequences of minimizers of $E_\eps$ for $\eps\to 0$.

\begin{proposition}\label{prop:compactness}
Suppose $2 \leq p<4$ and consider positive numbers $\{\eps_k\}_{k\in\N}\subset \R$ so that $\eps_\infty:=\lim_{k\to\infty} \eps_k$ exists and satisfies $0\leq \eps_\infty \ll 1$. Define
\[
I:=
\inf\limits_{\substack{\m\in \mathcal{M}\\Q(\m)=-1}} E_{\eps_\infty}(\m)
\quad
\begin{cases}
=4\pi, &\text{if }\eps_\infty=0,\\
<4\pi, &\text{if }\eps_\infty>0.
\end{cases}
\]
Moreover, let $\{\m_k\}_k\subset \mathcal{M}$ be asymptotically minimizing in the homotopy class $Q=-1$; that is, suppose that
\[ 
Q(\m_k)=-1 \quad \text{and} \quad
\lim_{k\to\infty} E_{\eps_k}(\m_k) = I.
\]
Finally, assume
\[
\liminf_{k\to\infty} \bigl(-H(\m_k)\bigr) > 0 
\quad \text{as well as} \quad
\limsup_{k\to\infty} \bigl(V(\m_k)-H(\m_k)\bigr) < \infty.
\]
Then, up to translations and a subsequence, there exists $\m_\infty\in \mathcal{M}$ with $Q(\m_\infty)=-1$ so that
\begin{align*}
\nabla \m_k \rightharpoonup \phantom{1-m_{3,\infty}} \mathllap{\nabla \m_\infty} &\quad\text{weakly in }L^2(\R^2),\\
%\phantom{\nabla} \m_k \stackrel{*}{\rightharpoonup} \phantom{\nabla} \m_\infty &\quad\text{weak-* in }L^\infty(\R^2),\\
 m_k \rightharpoonup \phantom{1-m_{3,\infty}} \mathllap{m_\infty} &\quad\text{weakly in }L^q(\R^2)\text{ for all }p\leq q < \infty,\\
1-m_{3,k} \rightharpoonup 1-m_{3,\infty} &\quad\text{weakly in }L^q(\R^2)\text{ for all }\tfrac{p}{2}\leq q <\infty,\\
\m_k \to \phantom{1-m_{3,\infty}} \mathllap{\m_\infty} &\quad\text{strongly in }L^q_\text{loc}(\R^2)\text{ for all }1\leq q < \infty,
\end{align*}
and
\[
\liminf_{k\to\infty} E_{\eps_k}(\m_k) \geq E_{\eps_\infty}(\m_\infty).
\]
In particular, the infimum $I$ is attained by $\m_\infty\in\mathcal{M}$.
%Moreover,
%\[ 
%\forall \delta>0 \colon \exists R>0 \colon \limsup_{k\to\infty}\int_{\R^2\setminus B_R} \lvert \nabla \m_k \rvert^2 \, dx < \delta.
%\]
%Moreover, $Q(\m_\infty)=0$ only if $\eps_\infty=0$ and $\m_\infty=\text{const}.$
\end{proposition}
%
%\begin{remark}
In the case $p=2$ with $\eps_\infty=0$, the above result does not apply to families of minimizers $\{\m_\eps\}_\eps$ of $E_\eps$, since we are unable to verify the bounds on $-H(\m_\eps)$ and $V(\m_\eps)$ as $\eps \to 0$ (in fact, in the given scaling, we expect $H(\m_\eps)\to 0$ as $\eps\to 0$). For $p=4$, on the other hand, the proof fails, since we cannot exclude ``vanishing'' in the concentration-compactness alternative -- in the derivation of Theorem~\ref{thm:1}, we will instead exploit the matching upper and lower bounds to $E_\eps$. %However, Theorem~\ref{thm:1} holds in either case.
%\end{remark}

Before turning to the proof of Proposition~\ref{prop:compactness}, however, we will deduce both Theorem~\ref{thm:1} and Theorem~\ref{thm:2}:
\begin{proof}[Proof of Theorem~\ref{thm:1}]
\medskip
\noindent \textbf{Step~1} (The case $p=4$): For $p=4$, we may appeal to the matching upper and lower bounds Lemma~\ref{lem:lowerbound}~and~\ref{lem:upperbound}. That is, 
\[ \m_\eps \colon \R^2 \to \St, \quad \m_\eps(x):=\Phi\bigl(\tfrac{x}{2(p-2)}\bigr), \]
is a minimizer of $E_\eps$ in the homotopy class $Q=-1$. Moreover, by Lemma~\ref{lem:lowerbound}, any minimizer $\tilde{\m}\in\mathcal{M}$ of $E_\eps$ must satisfy \eqref{eq:hel_el_kappa} for $\kappa=\tfrac{V_4(\tilde{\m})}{2V_4(\tilde{\m})}=\tfrac{1}{2}$.

%Setting $\kappa=1$ in \cite{Melcher_CSK}, the parameter $\eps$ corresponds to $h^{-1}$ after rescaling lengths by $h^{-1}\ll 1$. Lemmas~\ref{lem:lowerbound}~and~\ref{lem:upperbound} provide a lower and upper bound to $E_\eps(\m_\eps)$ as $\eps\to 0$.

\medskip
\noindent \textbf{Step~2} (The case $2\leq p<4$): % \LD{Check whether Prop. 1 also holds for $p=2$, so that we can include this case also here.}
When $V=V_p$ represents the classical Zeeman interaction, that is for $p=2$, the existence of a minimizer $\m_\eps$ of $E_\eps$ in the homotopy class $Q=-1$ has been shown in \cite{CSK}. However, the same approach can be used for the whole range $2\leq p < 4$:
Consider a minimizing sequence $\{\m_k\}_{k\in\N}\subset \mathcal{M}$ for $E_\eps$ with $Q(\m_k)=-1$, and let $0<\eps_k:=\eps\ll 1$.
%Once we prove
%\[
%\liminf_{k\to\infty} \bigl(-H(\m_k)\bigr) > 0 
%\quad \text{and} \quad
%\limsup_{k\to\infty} \bigl(V(\m_k)-H(\m_k)\bigr) < \infty,
%\]
%Proposition~\ref{prop:compactness} yields the claim.
%Indeed, the first estimate follows from 
Lemma~\ref{lem:upperbound} yields for $2<p<4$
\[
\lim_{k\to\infty} E_{\eps}(\m_k)=\inf\{E_\eps(\m) : \m\in\mathcal{M},\;Q(\m)=-1\} \leq 4\pi(1-2(p-2)\eps).
\]
Hence, using that $D(\m_k) + \eps V(\m_k)\geq 4\pi$ due to $Q(\m_k)=-1$, we obtain 
\begin{align*}
\MoveEqLeft \liminf_{k\to\infty} \Bigl(-\eps H(\m_k)\Bigr) = \liminf_{k\to\infty} \Bigl( D(\m_k)+\eps V(\m_k)- E_\eps(\m_k) \Bigr)\\
&\geq 4\pi- 4\pi(1-2(p-2)\eps)=8\pi(p-2)\eps>0.
\end{align*}
If $p=2$, we can use the upper bound $4\pi\bigl(1-\bigl(4+o(1)\bigr)\frac{\eps}{\lvert \ln \eps \rvert}\bigr)<4\pi$ to arrive at the same conclusion $\liminf_{k\to\infty} \bigl(-H(\m_k)\bigr)>0$.

On the other hand, we may use Lemma~\ref{lem:bound_norm} to obtain
\[ \sqrt{\eps} \limsup_{k\to\infty} \lvert H(\m_k) \rvert \stackrel{\eqref{eq:helicity_bound}}{\lesssim} \limsup_{k\to\infty} \bigl( D(\m_k) + \eps V(\m_k) \bigr) \stackrel{\text{Lem.~\ref{lem:bound_norm}}}{\lesssim} \limsup_{k\to\infty} E_\eps(\m_k) \leq 4\pi,  \]
i.e.
\[ \limsup_{k\to\infty} \bigl( V(\m_k)-H(\m_k) \bigr) < \infty.\]
Hence, we may apply Proposition~\ref{prop:compactness} to obtain convergence (up to a subsequence and translations) of $\{\m_k\}_{k\in\N}$ to a limit $\m_\infty\in\mathcal{M}$ with $Q(\m_\infty)=-1$ and
\[ I=\lim_{k\to\infty} E_\eps(\m_k) \geq E_\eps(\m_\infty) \geq I. \]
Thus, $\m_\infty$ minimizes $E_\eps$ in the class $\mathcal{M}$, subject to the constraint $Q=-1$.
By the $H^1$ continuity of the topological charge $Q(\m)$, the constrained minimizer $\m_\infty \in  \mathcal{M}$ constructed before is a local minimizer 
of $E_{\eps}(\m)$ in $\mathcal{M}$ and as such an almost harmonic map with an $L^2$ perturbation as considered in \cite{Moser_book} (see also \cite{Helein_book}).
Hence, $\m_\infty$ is H\"older continuous.
%applies and shows that $\m_\infty$ is H\"older continuous. 
\end{proof}

%As another consequence we prove compactness of sequences of magnetization configurations with energy excess $E_\eps-4\pi$ of order $\eps$ as $\eps \to 0$:

\begin{proof}[Proof of Theorem~\ref{thm:2}]
%\begin{corollary}\label{cor:conv_minimizers}
%Let $1<p\leq 2$ and choose $(\eps_k)_k\subset(0,\frac{1}{16}]$ so that $\lim_{k\to\infty}\eps_k=0$. Suppose that $(\m_k)_k\subset \mathcal{M}$ satisfies
%\begin{align*}
%Q(\m_k)=-1 \quad \text{and} \quad 4\pi-c\eps_k \leq E_{\eps_k}(\m_k) \leq 4\pi-C\eps_k \quad \forall k\in\N
%\end{align*}
%for some constants $c,C>0$. Then, there exists a limit $\m\in \mathcal{C}$ so that, up to a subsequence and translations, we have:
%\begin{align*}
%\nabla \m_k \to \phantom{1-m_{3,\infty}} \mathllap{\nabla \m_\infty} &\quad\text{strongly in }L^2(\R^2),\\
%%\phantom{\nabla} \m_k \stackrel{*}{\rightharpoonup} \phantom{\nabla} \m_\infty &\quad\text{weak-* in }L^\infty(\R^2),\\
%\phantom{\nabla \m_k} \mathllap{1-m_{3,k}} \rightharpoonup 1-m_{3,\infty} &\quad\text{weakly in }L^q(\R^2)\text{ for all }p\leq q \leq 2p,\\
%\phantom{\nabla} \m_k \to \phantom{1-m_{3,\infty}} \mathllap{\m_\infty} &\quad\text{strongly in }L^q_\text{loc}(\R^2)\text{ for all }1\leq q < \infty.
%\end{align*}
%\end{corollary}
By the lower bound Lemma~\ref{lem:lowerbound}, we may assume w.l.o.g. that the constant $0<C_0<\infty$ satisfies
\begin{align}\label{eq:boundsexcess}
4\pi-C_0^{-1}\eps \leq E_\eps(\m_\eps) \leq 4\pi - C_0\eps.
\end{align}

\medskip
\noindent \textbf{Step~1} (Verification of the assumptions of Proposition~\ref{prop:compactness}): \textit{We prove
\[ \lim_{\eps\to 0} D(\m_\eps) = 4\pi, \quad \liminf_{\eps\to 0} \bigl( -H(\m_\eps)\bigr) >0 \; \text{and} \; \limsup_{\eps\to 0} \bigl( V(\m_\eps)-H(\m_\eps) \bigr) < \infty. \]
}

%Since $E_{\eps}(\m_\eps)<4\pi$, while $D(\m_\eps)\geq 4\pi$, we have $H(\m_\eps)\leq H(\m_\eps)+V(\m_\eps)<0$.
Indeed, we have
\[ 
 -H(\m_\eps) = \tfrac{1}{\eps}\Bigl( D(\m_\eps)+\eps V(\m_\eps) - E_{\eps}(\m_\eps) \Bigr)  \stackrel{\eqref{eq:boundsexcess}}{\geq} \tfrac{1}{\eps} \Bigl( 4\pi - (4\pi-C_0\eps) \Bigr) = C_0,
\]
so that $\liminf_{\eps\to 0}\bigl(-H(\m_\eps)\bigr)>0$. On the other hand, Lemma~\ref{lem:lowerbound} and the topological lower bound yield
\[ 4\pi \leq D(\m_\eps) \leq \tfrac{1}{1-4\eps} E_\eps(\m_\eps) \stackrel{\eqref{eq:boundsexcess}}{\leq} \tfrac{4\pi-C_0\eps}{1-4\eps} \to 4\pi \quad \text{as }\eps\to 0. \]
%$\tfrac{1}{2} D(\m_\eps)\leq E_{\eps}(\m_\eps) \leq 4\pi$, which (again by Lemma~\ref{lem:lowerbound}) improves to
%\[ 
%D(\m_\eps) \leq \underbrace{(1-8\eps) D(\m_\eps)}_{\leq E_{\eps}(\m_\eps) \leq 4\pi} + \underbrace{8\eps D(\m_\eps)}_{\leq 64\pi \eps} \leq 4\pi+64\pi\eps \quad \forall 0<\eps\ll 1.
%\]
%Since $Q(\m_\eps)=-1$, we have $4\pi \leq D(\m_\eps) \leq 4\pi + 64\pi \eps$, i.e. 
Hence, $D(\m_\eps)\to 4\pi$ for $\eps \to 0$.

Due to \eqref{eq:helicity_bound}, it remains to prove that $V(\m_\eps)$ is bounded uniformly in $0<\eps\ll 1$. Indeed, from Lemma~\ref{lem:bound_norm}, we obtain
\[ 
\tfrac{\eps}{2} V(\m_\eps) \leq \underbrace{E_\eps(\m_\eps)}_{\stackrel{\eqref{eq:boundsexcess}}{\leq} 4\pi} - 4\pi(1-16\eps) \leq 64\pi \eps \quad \forall 0<\eps \ll 1.
\]
Thus, $\limsup_{k\to\infty} \bigl(V(\m_k)-H(\m_k)\bigr) < \infty$.
%$, Lemma~\ref{lem:lowerbound} yields $\limsup_{k\to\infty}D(\m_k) < \infty$. %Moreover, the Pohozaev scaling argument yields $-H(\m_\eps)=2V(\m_\eps)$. In particular, it suffices to prove that $V(\m_\eps)$ remains bounded.  Moreover,
%\[
%D(\m_k)-c\eps_k\leq E_{\eps_k}(\m_\eps)=D(\m_\eps)+\eps \bigl(H(\m_k)+ V(\m_k)\bigr) \leq 4\pi - \tfrac{1}{C}\eps.
%\]
%Thus,
%\[ 
%\tfrac{1}{C} \leq \eps^{-1}\bigl(D(\m_\eps)-4\pi\bigr) + \tfrac{1}{C} \leq V(\m_\eps) \leq C. \qedhere
%\]

\medskip
\noindent \textbf{Step~2} (Proof of part i)):
By Step~1, we may apply Proposition~\ref{prop:compactness}. Hence, there exists $\m_0\in \mathcal{M}$ with $Q(\m_0)=-1$ so that in the limit $\eps \to 0$, along a subsequence and up to translations (not relabeled):
%\begin{align*}
%\nabla \m_\eps \rightharpoonup \phantom{1-m_{3,0}} \mathllap{\nabla \m_0} \quad &\text{in $L^2(\R^2)$},\\
%\m_\eps \rightharpoonup \phantom{1-m_{3,0}} \mathllap{\m_0}\quad &\text{in $L^{2p}(\R^2)$},\\
%1-m_{3,\eps} \rightharpoonup 1-m_{3,0} \quad &\text{in $L^p(\R^2)$}.
%\end{align*}
\begin{align*}
\nabla \m_\eps \rightharpoonup \phantom{1-m_{3,0}} \mathllap{\nabla \m_0} &\quad\text{weakly in }L^2(\R^2),\\
%\phantom{\nabla} \m_k \stackrel{*}{\rightharpoonup} \phantom{\nabla} \m_\infty &\quad\text{weak-* in }L^\infty(\R^2),\\
 m_\eps \rightharpoonup \phantom{1-m_{3,0}} \mathllap{m_0} &\quad\text{weakly in }L^q(\R^2)\text{ for all }p\leq q < \infty,\\
1-m_{3,\eps} \rightharpoonup 1-m_{3,0} &\quad\text{weakly in }L^q(\R^2)\text{ for all }\tfrac{p}{2}\leq q <\infty,\\
\m_\eps \to \phantom{1-m_{3,0}} \mathllap{\m_0} &\quad\text{strongly in }L^q_\text{loc}(\R^2)\text{ for all }1\leq q < \infty,
\end{align*}
Since, by Step~1 and $Q(\m_0)=-1$, we have $4\pi = \liminf_{\eps\to 0} D(\m_\eps) \geq D(\m_0) \geq 4\pi$, weak convergence $\nabla \m_\eps \rightharpoonup \nabla \m_0$ upgrades to strong convergence in $L^2(\R^2)$. In particular, $\m_0\in \mathcal{C}$, which proves the first part of the claim.

\medskip
\noindent \textbf{Step~3} (Proof of part ii)):
Assume that
\[ E_\eps(\m_\eps) \leq 4\pi+\eps \min_{\m\in\mathcal{C}} \bigl( H(\m)+V(\m)\bigr) + o(\eps) \]
holds as $\eps \to 0$, i.e.
\[ \limsup_{\eps\to 0} \eps^{-1}\bigl(E_\eps(\m_\eps)-D(\m_0)\bigr) \leq \min_{\m\in\mathcal{C}} \bigl(H(\m)+V(\m)\bigr). \]
By Step~2, we have $\nabla \m_\eps \to \nabla \m_0$ strongly in $L^2(\R^2)$ and $1-m_{3,\eps}\rightharpoonup 1-m_{3,0}$ weakly in $L^\frac{p}{2}(\R^2)$ and $L^2(\R^2)$ along a suitable subsequence as $\eps \to 0$. Thus, we obtain
\[ \lim_{\eps\to 0} H(\m_\eps) = H(\m_0), \qquad  \liminf_{\eps\to 0} V(\m_\eps) \geq V(\m_0), \]
and, using that $D(\m_\eps)\geq 4\pi = D(\m_0)$,
\[ \liminf_{\eps\to 0} \eps^{-1}\bigl(E_\eps(\m_\eps)-D(\m_0)\bigr) \geq H(\m_0)+V(\m_0) \geq \min_{\m\in\mathcal{C}} \bigl( H(\m)+V(\m) \bigr). \]
Therefore,
\[ \eps^{-1}\bigl(E_\eps(\m_\eps)-D(\m_0)\bigr) \to \min_{\m\in\mathcal{C}} \bigl( H(\m)+V(\m) \bigr) = H(\m_0)+V(\m_0)  \quad \text{as } \eps\to 0. \]
In particular, we obtain
\[ \lim_{\eps\to 0}\eps^{-1}\bigl( D(\m_\eps)-D(\m_0) \bigr) = 0 \quad \text{and} \quad \lim_{\eps\to 0} V(\m_\eps) = V(\m_0). \]
Hence, $\m_\eps\to\m_0$ strongly in $L^p(\R^2)$, i.e. $d(\m_\eps,\m_0) \to 0$ as $\eps \to 0$, up to translations and a suitable subsequence.

%Finally, having 
%\[ \liminf_{\eps\to 0} V(\m_\eps) > V(\m_0) \]
%would yield the contradiction 
%\[ \liminf_{\eps\to 0} E_\eps(\m_\eps) > 4\pi. \quad\lightning \]
%Thus, we also have $V(\m_\eps) \to V(\m_0)$, so that $\m_\eps \to \m_0$ strongly in $L^p(\R^2)$. \LD{Find reference.} In particular, $d(\m_\eps,\m_0) \to 0$ for $\eps \to 0$ and

%
%{\color{red}(Question: Why $\m_\eps \not\equiv \m_0$?)}

%
%\begin{lemma}
%$
%\displaystyle{\m_0 \in \mathrm{argmin}\left( H(\m)+V(\m): \m \in \mathcal{C} \right)}
%$
%\end{lemma}
%
%\begin{proof}
%Let $F=H+V$ and suppose there exists $\delta>0$ such that 
%\[
%F(\m_0)>\inf_{\m \in \mathcal{C}} F(\m)+\delta
%\quad
%\text{hence}
%\quad
%E_\eps(\m_0)> \inf_{\m \in \mathcal{C}} E_\eps(\m) + \eps \delta.
%\]
%If we choose $\eps$ so small that $|F(\m_0)-F(\m_\eps) |< \delta$, then
%\begin{eqnarray*}
%E_\eps(\m_\eps)
% \ge D(\m_\eps)+ \eps \left( F(\m_0) -  \delta \right)  
%\ge E_\eps(\m_0) -  \eps \delta > \inf_{\m \in \mathcal{C}} E_\eps(\m),
%\end{eqnarray*}
%which contradicts the minimality of $E_\eps(\m_\eps)$.
%\end{proof}

Recall that (with the identification $\R^2\simeq \C$) $\m \in \mathcal{C}$ may be represented as
\[
 \m(x)=\m^{(\rho,\varphi)}(x) = \Phi(ax+b) %
 = e^{i \varphi} \Phi(\rho x  + \tilde{b}) 
\]
for two complex numbers $a=\rho e^{i \varphi}\neq 0$ and $b$, with $\tilde{b}=a^{-1} b$. Thus, dropping $b$ due to the translation invariance of the problem, minimization is a finite dimensional problem; in fact, we have
\[ H(\m^{(\rho,\varphi)})+ V(\m^{(\rho,\varphi)}) = \tfrac{\cos\varphi}{\rho} H(\Phi) + \tfrac{V(\Phi)}{\rho^2} = -\tfrac{8\pi \cos\varphi}{\rho} + \tfrac{2\pi}{\rho^2(p-2)}, \]
which obviously is minimized by $\varphi\in 2\pi \Z$ and $\rho=\tfrac{1}{2(p-2)}$.
Hence, up to translation, the unique minimizer of $H+V$ in $\mathcal{C}$ is given by
\[
\m_0(x)= \Phi(\rho  x) \quad \text{with} \quad \rho = \tfrac{1}{2(p-2)} = -2 \tfrac{V(\Phi)}{H(\Phi)}.
\]
In particular, the whole sequence $\{\m_\eps\}_{\eps>0}$ converges with respect to $d$, up to translations, to the unique limit $\m_0$.
\end{proof}

It remains to prove Proposition~\ref{prop:compactness}:
\begin{proof}[Proof of Proposition~\ref{prop:compactness}]
We first remark that in view of \eqref{eq:helicity_bound} and Lemma~\ref{lem:bound_norm}, the assumptions also imply 
\[ 
\limsup_{k\to\infty} D(\m_k) < \infty \quad\text{and}\quad \liminf_{k\to\infty} V(\m_k) > 0.
\]
Moreover, we will use the symbol $\lesssim$ to indicate that an inequality holds up to a universal, multiplicative constant that may change from line to line.

\medskip
\noindent \textbf{Step~1}: \textit{We prove:
\begin{align*}
\lvert H(\m_k) \rvert
&\lesssim \Bigl( \sup_{y\in \R^2} \Bigl( \int_{B_1(y)} \lvert \nabla \m_k \rvert^2 \, dx\Bigr)^{\frac{1}{2}} + \sup_{y\in\R^2} \Bigl( \int_{B_1(y)} \lvert \nabla \m_k \rvert^2 \, dx \Bigr)^{\frac{2}{p}-\frac{1}{2}} \Bigr)\\
&\qquad\times \Bigl(D(\m_k)+V(\m_k)\Bigr) \qquad \forall k\in\N.
\end{align*}}
Indeed, choose $\delta>0$ so that $\cup_{y\in\delta\Z^2} B_1(y)=\R^2$. Then, we have
\begin{align*}
\Bigl\lvert \int_{\R^2} (1-m_{3,k}) (\nabla \times m_k) \, dx \Bigr\rvert
\lesssim \sum_{y\in \delta \Z^2} \Bigl( \int_{B_1(y)} \underbrace{(m_{3,k}-1)^2}_{=\frac{1}{4}\lvert \m_k - \ein_3 \rvert^4} \, dx \Bigr)^{\frac{1}{2}} \Bigl( \int_{B_1(y)} \lvert \nabla \m_k \rvert^2 dx \Bigr)^{\frac{1}{2}}.
%&\leq \sup_{y\in\R^2} \Bigl(\int_{B_1(y)} \lvert \nabla \m_k \rvert^2 dx \Bigr)^{\frac{1}{p}-\frac{1}{2}}.
\end{align*}
Moreover, the Sobolev embedding theorem and Jensen's inequality yield
\begin{align*}
\Bigl( \int_{B_1(y)}\lvert \m_k - \ein_3 \rvert^4 \, dx \Bigr)^{\frac{1}{2}} &\lesssim \int_{B_1(y)} \lvert \nabla \m_k \rvert^2 + \lvert \m_k - \ein_3 \rvert^2 \, dx\\
&\lesssim \int_{B_1(y)} \lvert \nabla \m_k \rvert^2 dx + \Bigl( \int_{B_1(y)} \tfrac{1}{2^p} \lvert \m_k - \ein_3 \rvert^p \, dx \Bigr)^{\frac{2}{p}},
\end{align*}
so that, using Young's inequality in the last step,
\begin{align*}
\lvert H(\m_k) \rvert
&\lesssim \sum_{y\in\delta \Z^2} \Bigl( \int_{B_1(y)} \lvert \nabla \m_k \rvert^2 \,dx \Bigr)^\frac{3}{2}\\
&\quad+ \sum_{y\in\delta \Z^2}\Bigl( \int_{B_1(y)} \lvert \nabla \m_k \rvert^2 \, dx \Bigr)^{\frac{1}{2}} \Bigl( \int_{B_1(y)} \tfrac{1}{2^p}\lvert \m_k - \ein_3 \rvert^p \, dx\Bigr)^{\frac{2}{p}}\\
&\leq \sup_{y\in \R^2} \Bigl( \int_{B_1(y)} \lvert \nabla \m_k \rvert^2 \, dx\Bigr)^{\frac{1}{2}} D(\m_k)\\
&\quad + \sup_{y\in\R^2} \Bigl( \int_{B_1(y)} \lvert \nabla \m_k \rvert^2 \, dx \Bigr)^{\frac{2}{p}-\frac{1}{2}} \bigl( D(\m_k)+V(\m_k)\bigr),
\end{align*}
which is the claim.

\medskip
\noindent \textbf{Step~2} (Concentration-compactness):
We consider the full energy density \eqref{eq:density} to define
$
\rho_k:=e_\eps(\m_k) \ge 0.
%=\tfrac{1}{2}\lvert \nabla \m_k \rvert^2 + \eps_k \Bigl( (\m_k-\ein_3) \cdot \nabla \times \m_k + \tfrac{1}{2^p}\lvert \m_k - \ein_3 \rvert^p \Bigr) \geq 0, \quad k\in\N.
$
Note that we have
\[
\rho_k \gtrsim \lvert \nabla \m_k \rvert^2 + \eps_k\tfrac{1}{2^p}\lvert \m_k - \ein_3 \rvert^p \quad \forall k\in\N
\]
and
\[
\lim_{k\to\infty}\int_{\R^2} \rho_k \, dx = I >0.
\]
Hence, we may apply the concentration-compactness lemma (see, e.g., \cite{Lions_CC_L1}) to the sequence $\{\rho_k\}_{k\in\N}$ of non-negative densities and obtain that, up to a subsequence, one of the following holds:
\begin{itemize}
\item \textit{Compactness}: There exists a sequence $\{y_k\}_{k\in\N}\subset\R^2$ so that
\[
\forall \delta >0 \colon \exists R < \infty \colon
  \int_{\R^2\setminus B_R(y_k)} \rho_k \, dx \leq \delta.
\]
\item \textit{Vanishing}: We have
\[
\lim_{k\to\infty} \sup_{y\in\R^2} \int_{B_R(y)} \rho_k \, dx = 0 \quad \forall R<\infty.
\]
\item \textit{Dichotomy}: There exist $a^{(1)},a^{(2)}>0$ so that $a^{(1)}+a^{(2)}=I$ and for all $\delta>0$, there exist $k_0\in \N$, $\{y_k\}_{k\in\N}\subset \R^2$, $R<\infty$, and a sequence $R_k\to\infty$, so that for $k\geq k_0$:
\[
\Bigl\lvert a^{(1)}-\int_{B_R(y_k)} \rho_k \, dx \Bigr\rvert +
\Bigl\lvert a^{(2)}-\int_{\R^2\setminus B_{R_k}(y_k)} \rho_k \, dx \Bigr\rvert +
\Bigl\lvert \int_{B_{R_k}(y_k)\setminus B_R(y_k)} \rho_k \, dx \Bigr\rvert
\leq \delta.
\]
\end{itemize}
In order to conclude, we need to rule out vanishing and dichotomy.

\medskip
\noindent \textbf{Step~2a} (Ruling out ``Vanishing''):
Suppose vanishing holds. Since $\rho_k$ controls $\tfrac{1}{2}\lvert \nabla \m_k \rvert^2$, while $V(\m_k)$ is bounded by assumption, Step~1 yields $\lim_{k\to\infty} H(\m_k) = 0$, contradicting the assumption $\liminf_{k\to\infty} \bigl(-H(\m_k)\bigr)>0$.

\medskip
\noindent \textbf{Step~2b} (Ruling out ``Dichotomy''):
Suppose dichotomy holds. In particular, for fixed $0<\delta\ll 1$ (to be specified later), we have
\begin{align*}
\int_{B_{R_k}\setminus B_R} \lvert \nabla \m_k \rvert^2+\eps_k \tfrac{1}{2^p}\lvert \m_k -\ein_3 \rvert^p dx \lesssim \int_{B_{R_k}\setminus B_R} \rho_k \, dx \leq \delta.
\end{align*}
W.l.o.g., we may assume that $R^2\delta^{\frac{p-2}{2}}\geq 1$ and $k\gg 1$, so that $R_k\geq 4R$.

If $\eps_\infty=0$, we may apply Lemma~\ref{lem:cutoff} with $\sigma=0$, otherwise with $\sigma=1$, and define $\m_k^{(i)}\in\mathcal{M}$, $i=1,2$, so that for some constant $C(\delta,R)$ and $c_k\in [R,2R]$:
\begin{alignat*}{2}
\m_k^{(1)}&=\m_k \quad \text{on } B_{c_k},               &\qquad V(\m_k^{(1)})&\lesssim C(\delta,R),\\
\m_k^{(2)}&=\m_k \quad \text{on }\R^2\setminus B_{2c_k}, &\qquad V(\m_k^{(2)})&\lesssim V(\m_k)+C(\delta,R),
\end{alignat*}
and
\begin{align*}
\MoveEqLeft\int_{\R^2\setminus B_{c_k}} \lvert \nabla \m_k^{(1)} \rvert^2 + \sigma \,\eps_k \tfrac{1}{2^p}\lvert \m_k^{(1)}-\ein_3 \rvert^p \, dx\\
&+ \int_{B_{2c_k}} \lvert \nabla \m_k^{(2)} \rvert^2 + \sigma \, \eps_k \tfrac{1}{2^p} \lvert \m_k^{(2)}-\ein_3 \rvert^p \, dx\\
&\lesssim \delta + \sigma (\tfrac{\delta}{R^2})^{2/p} \lesssim \delta.
\end{align*}
In particular, we have
\begin{align*}
\MoveEqLeft \bigl\lvert Q(\m_k^{(1)}) + Q(\m_k^{(2)}) - Q(\m_k) \bigr\rvert \leq \Bigl\lvert \tfrac{1}{4\pi}\int_{B_{2c_k}\setminus B_{c_k}} \omega(\m_k) \, dx \Bigr\rvert\\
&+ \Bigl\lvert Q(\m_k^{(1)})-\tfrac{1}{4\pi} \int_{B_{c_k}} \omega(\m_k) \, dx \Bigr\rvert + \Bigl\lvert Q(\m_k^{(2)})-\tfrac{1}{4\pi} \int_{\R^2\setminus B_{2c_k}} \omega(\m_k) \, dx \Bigr\rvert \lesssim \delta.
\end{align*}
Hence, since $Q(\m_k)=-1$ and $Q(\m_k^{(i)})\in \Z$, $i=1,2$, we obtain
\[
Q(\m_k^{(1)})+Q(\m_k^{(2)})=Q(\m_k)=-1.
\]
Moreover, using the estimate 
\[
\lvert H(\m_k^{(i)}) \rvert \stackrel{\eqref{eq:helicity_bound}}{\lesssim} \Bigl( D(\m_k^{(i)}) \smash{\overbrace{V(\m_k^{(i)})}^{\lesssim C <\infty}} \Bigr)^{\frac{1}{2}},
\]
which also holds localized to $B_{2c_k}$ and $\R^2\setminus B_{c_k}$, respectively,
the ``dichotomy'' condition yields
\begin{align}\label{eq:energy_cutoff}
E(\m_k^{(i)}) \leq a^{(i)}+C\sqrt{\delta} < I \leq 4\pi \quad \text{if }\delta\ll 1, \text{ for }i=1,2.
\end{align}
If $\lvert Q(\m_k^{(i)})\rvert \geq 2$ for some $i\in\{1,2\}$, Lemma~\ref{lem:lowerbound} and the inequality $D(\m)\geq 4\pi \lvert Q(\m) \rvert$ imply
\[
4\pi>E(\m_k^{(i)}) \stackrel{\text{Lem.~\ref{lem:bound_norm}}}{\geq} (1-16\eps) D(\m_k^{(i)}) \geq 3\pi \lvert Q(\m_k^{(i)}) \rvert \geq 6\pi\quad \text{if }0< \eps \ll 1.\quad \lightning
\]
Moreover, $Q(\m_k^{(i)})=1$ for some $i\in\{1,2\}$ yields $Q(\m_k^{(3-i)})=-2$, which leads to the same contradiction as above.

Thus, we have $Q(\m^{(i)}_k)\in\{-1,0\}$ for $i=1,2$, i.e. there exists $i_0\in \{1,2\}$ with $Q(\m^{(i_0)}_k)=-1$.

If $\eps_\infty>0$, we directly obtain a contradiction, since $\m_k^{(i_0)}$ is admissible in the variational problem $I$, hence
\[ I\leq E(\m_k^{(i_0)}) \stackrel{\eqref{eq:energy_cutoff}}{<} I. \quad \lightning \]
If $\eps_\infty=0$, we use that $H(\m_k^{(i_0)})+V(\m_k^{(i_0)})$ remains bounded by construction (see Lemma~\ref{lem:cutoff} and \eqref{eq:helicity_bound}, and note that $R$ and hence also $C(\delta,R)$ depend on $\delta$, but not on $k$), and thus 
\[
4\pi \stackrel{\eqref{eq:energy_cutoff}}{>} a^{(i_0)}+C\sqrt{\delta} \geq \liminf_{k\to\infty} E_{\eps_k}(\m_k^{(i_0)}) \geq \liminf_{k\to\infty} D(\m_k^{(i_0)}) \geq 4\pi. \quad \lightning
\]
Therefore, dichotomy cannot occur.

\medskip
\noindent \textbf{Step~3} (Conclusion):
By Step~2, we may assume that compactness holds in the concentration-compactness alternative. W.l.o.g., $y_k=0$ for all $k\in\N$. By passing to a subsequence and using Rellich's theorem, we may assume that there exists $\m_\infty\in \dot{H}^1(\R^2;\St)$ such that
\begin{align*}
\phantom{1-m_{3,k}}\mathllap{\nabla \m_k} \rightharpoonup \phantom{1-m_{3,\infty}}\mathllap{\nabla \m_\infty} \quad &\text{weakly in }L^2(\R^2),\\
\phantom{1-m_{3,k}}\mathllap{m_k} \rightharpoonup \phantom{1-m_{3,\infty}}\mathllap{m_\infty} \quad &\text{weakly in }L^p(\R^2),\\
1- m_{3,k} \rightharpoonup 1-m_{3,\infty} \quad &\text{weakly in }L^\frac{p}{2}(\R^2),\\
\phantom{1-m_{3,k}}\mathllap{\m_k} \stackrel{\star}{\rightharpoonup} \phantom{1-m_{3,\infty}}\mathllap{\m_\infty} \quad &\text{weak-* in }L^\infty(\R^2),\\
\phantom{1-m_{3,k}}\mathllap{\m_k} \to \phantom{1-m_{3,\infty}}\mathllap{\m_\infty} \quad &\text{strongly in }L^p_\text{loc}(\R^2) \text{ for }1\leq p < \infty.
\end{align*}
Since compactness holds, we have (see \cite[Lemma~4.1]{CSK})
\[ 
I=\liminf_{k\to\infty} \bigl( E_{\eps_k}(\m_k) + 4\pi Q(\m_k) \bigr) +4\pi \geq E_{\eps_\infty}(\m_\infty) + 4\pi Q(\m_\infty) + 4\pi.
\]
If $\eps_\infty>0$, i.e. $I<4\pi$, we may immediately exclude $Q(\m_\infty) \geq 0$. On the other hand, Lemma~\ref{lem:bound_norm} in form of the inequality $E_{\eps_\infty}(\m_\infty) \geq 4\pi(1-16\eps_\infty) \lvert Q(\m_\infty) \rvert$ rules out $\lvert Q(\m_\infty)\rvert\geq 2$, if $\eps_\infty$ is sufficiently small. Hence, we have $Q(\m_\infty)=-1$.% so that $\m_\infty$ minimizes $E$ subject to $Q(\m)=-1$.

If $\eps_\infty=0$, i.e. $I=4\pi$, we may argue similarly to obtain $Q(\m_\infty)\in\{-1,0\}$. Moreover, if $Q(\m_\infty)=0$, we obtain $E_{\eps_\infty}(\m_\infty)=D(\m_\infty)=0$, i.e. $\m_\infty = \text{const}$. In particular, using the ``compactness'' condition and the initial assumption $\limsup_{k\to\infty} V(\m_k) < \infty$ to reduce the problem to a bounded set, we obtain $H(\m_k)\to 0$. $\lightning$ Hence, also for $\eps_\infty=0$, we have $Q(\m_\infty)=-1$.
\end{proof}

%\begin{lemma}[Boundedness of first-order terms]
%Suppose $(\m_\eps)_\eps$ is a sequence of minimizers of $E_\eps$. Then, there exists a constant $0<C<\infty$ such that
%\[
%\tfrac{1}{C} \leq -H(\m_\eps) \leq C,
%\qquad
%\tfrac{1}{C} \leq V(\m_\eps) \leq C.
%\]
%\end{lemma}

%\section{Expansion of $E_\eps(\m_\eps)$}

%\begin{proposition}
%Suppose $(\m_\eps)_\eps\subset \mathcal{M}$ is a sequence of minimizers of $E_\eps$ for $0<\eps\ll 1$. Then, up to translations, there exists $\m_0\in\mathcal{C}$ so that $d(\m_\eps,\m_0) \to 0$ for $\eps\to 0$ and
%\begin{align}\label{eq:expEeps}
%E_\eps(\m_\eps) = D(\m_0) + \eps \bigl(H(\m_0)+V(\m_0)\bigr) + o(\eps).  
%\end{align}
%In fact, $\m_0=\Phi(\frac{2^{p-1}}{p-1} x)$, so that $Q(\m_0)=-1$, $D(\m_0)=4\pi$ and
%\[ H(\m_0)+V(\m_0)=\min_{\m\in \mathcal{C}} \bigl(H(\m) + V(\m)\bigr) = 4\pi(p-1). \]
%\end{proposition}

\section{Regularity of the dynamic problem and proof of Theorem \ref{thm:3}}\label{sec:dynamics}

Let us now consider the pulled back Landau-Lifshitz-Gilbert equation
\[
(\del_t - \nu \del_z) \m = \m \times \Big[  \alpha  (\del_t-\nu \del_z) \m -  \bs h_\eps(\m)  \Big]
\]
on $\R^2 \times [0,T]$ as motivated in the introduction. The effective field reads
\[
\bs h_{\rm eff}(\m)= \Delta \m - \eps \left( 2 \nabla \times \m + \bs f(\m) \right).
\]
According to our choice of potential energy we have for $p=4$
\[
\bs f(\m)= \frac{1}{4}|\m-\ein_3|^{2} (\m -\ein_3)
\]
which is smooth. 

\subsection*{Local well-posedness}
Starting from spatial discretization as in \cite{Sulem_Sulem_Bardos:86, Alouges_Soyeur:92, Carbou_Fabrie_R3:01} or spectral truncation as in \cite{Melcher:11, Taylor}
one obtains for initial conditioins $\m^{0} \in \mathcal{M}$ such that $\nabla \m^{0} \in H^2(\R^2)$ 
a local solution $\m:\R^2 \times [0, T^\ast) \to \St$ for some terminal time $T^\ast >0$, which is bounded below in terms of $\| \nabla \m^{0} \|_{H^2}$, such that 
for all $T <T^\ast$
\[
E_\eps(\m) \in L^\infty(0,T) \quad \text{and} \quad \nabla \m \in L^\infty\left(0,T; H^2(\R^2)\right)  \cap L^2\left(0,T; H^3(\R^2)\right).
\]
Initial data $\m^0$ and $\nabla \m^0$ are continuously attained in $\mathcal{M}$ and $H^2(\R^2)$, respectively, see \cite{Taylor}.
As $\nabla \m \in W^{1, \infty}\left(0,T; L^2(\R^2)\right)$,
interpolation and Sobolev embedding yield uniform H\"older continuity of $\nabla \m$ in $\R^2 \times [0,T]$. 
Uniqueness in this class can be shown by means of a Gronwall argument as in \cite{Melcher:11, Taylor}. Due to the slow decay of $\m-\ein_3$,
the conventional $L^2$-distance is replaced by a suitably weighted $L^2$-distance, e.g.
\[
\|\bs u\|_{L^2_\ast}^2 := \int_{\R^2} \frac{|\bs u(x)|^2}{1+|x|^2} \, dx \lesssim \|\bs u\|_{L^4}^2
\]
%and a Gronwall argument as in \cite{CM:Cauchy}
%\[
%\| (\m_1-\m_2)(T) \|_{L^2_\ast}^2  \le \exp \left( c \sum_{i=1,2} \int_0^T \| \nabla \m_i(s)\|_{L^\infty} \, ds \right) \| (\m_1-\m_2)(0) \|_{L^2_\ast}^2
%\]
%for a universal constant $c>0$ and all $0<T<T^\ast$.
As $\nabla \m(t) \in H^3(\R^2)$ for almost every $t <T^\ast$,
uniqueness and a bootstrap argument imply $\nabla \m \in L^\infty_{\rm loc}(0,T^\ast; H^k(\R^2))$ for arbitrary $k\in N$, in particular $\m$ is smooth. 
Now one may deduce the following Sobolev estimate (which equally holds true for approximate equations)
%Moreover
%\[
%\sup_{ 0 \le t \le T} \int_{|x|>R} e_\eps(\m(t)) \, dx \to 0 \quad \text{as} \quad R \to \infty
%\]
\begin{align*}
\sup_{0 \le t \le T} \|\nabla \bs m(t)\|_{H^k}^2 +& \int_{0}^{T}  \|\nabla \bs m (t)\|_{H^{k+1}}^2 \; dt \\
&\le c \, \Big(1+ \sup_{t \in [0,T]} \|\nabla \bs m(t)  \|_{L^\infty}^2\Big)\int_{0}^{T}  \|\nabla \bs m(t)\|_{H^{k}}^2 \; dt
\end{align*}
for all $0\le k \le 2$ and $0<T<T^\ast$ (cf. Lemma \ref{lemma:local_Sobolev} below). Hence, if $T^\ast < \infty$, then 
\[
\displaystyle{\limsup_{t \nearrow T^\ast} \|\nabla \m(t)\|_{L^\infty}=\infty.}
\]
\subsection*{Local Sobolev estimates}
Due to lower order perturbations, \eqref{eq:LLG_MF}  is translation- 
but not dilation-invariant. However, with respect to transformations
$
\tilde {\bs m}(x,t)=\m(x_0+\lambda x, t_0+ \lambda^2 t)
$
the parameters $\eps$ and $\nu$ exhibit the following scaling behavior 
$
\tilde \eps= \lambda \eps$ and $\tilde \nu=\lambda \nu
$
while $\tilde{\bs f}(\m)= \lambda \bs f(\m)$. 
Hence, the coefficients of the lower order perturbations are uniformly bounded in the blow-up regime $\lambda \le 1$. 
In this case we shall call $\tilde{\bs m}=\bs m$ a blow-up solution. We shall need a localized version of the a priori estimates
from \cite{Melcher:11} that led to the existence result. Here and in the sequel let
\[
P_R=B_R \times (-R^2,0),
\]
the parabolic cylinder in space-time $\R^2 \times (-\infty,0]$.
\begin{lemma} \label{lemma:local_Sobolev}
Suppose $0 \le k  \le 2$ and $\m$ is a blow-up solution in a neighborhood of $\overline{P_R}$ for some $R\ge 1$. Then
\begin{align*}
 \|\nabla \bs m(0)\|_{H^k(B_{R/2})}^2 + \int_{-(R/2)^2}^{0}  \|\nabla \bs m (t)\|_{H^{k+1}(B_{R/2})}^2 \; dt \\
\le c \, \Big(1+\|\nabla \bs m  \|_{L^\infty(P_R)}^2\Big)\int_{-R^2}^{ 0}  \|\nabla \bs m(t)\|_{H^{k}(B_{R})}^2 \; dt
\end{align*}
for a constant $c$ that only depends on the parameters $\alpha$, $\nu$, $\eps$. In particular,
\[
%\esssup_{-(R/8)^2 \le t \le 0}  \|\nabla \bs m(t)\|_{L^\infty(B_{R/8})}^2 
|\nabla \m(0,0)|^2 \le C \int_{-R^2}^{ 0}  \|\nabla \bs m(t)\|_{L^{2}(B_{R})}^2 \; dt
\]
for a constant $C$ that only depends on the parameters $\alpha$, $\nu$, $\eps$ and $\|\nabla \bs m  \|_{L^\infty(P_R)}$.
\end{lemma}

\begin{proof}[Sketch of proof] The Landau-Lifshitz form of the equation reads
\[
(1+ \alpha^2) \, \del_t \m =\alpha \left( \Delta \m + |\nabla \m|^2 \m \right) - \nabla \cdot \left( \m \times \nabla \m \right) +\bs F(\nabla \m ,\m),
\]
for a smooth tangent field $\bs F$ that is linear in $\nabla \m$. 
The standard procedure uses test functions $\del^{\bs \nu} ( \phi^2 \del^{\bs \nu} \bs m)$, where $\bs \nu$ is a multi index of length $1 \le |\bs \nu| \le k+1$,
and
$\phi(x,t)=\varphi(x) \eta(t)$ is an appropriate space-time cut-off function $0\le \phi \le 1$ where $\varphi \in C^\infty_0(B_1)$ with
$\varphi|_{B_{1/2}}=1$ and $\eta \in C^\infty(\R)$ with $\eta(t)=0$ for $t<-1$ and $\eta(t)=1$ for $t>-1/4$. In the case $R>1$ one uses suitable rescalings
of $\varphi$ and $\eta$. Let us only estimate the contribution from the non-coercive term of second
order $\nabla \cdot( \bs m \times \nabla  \m)$:
\[
I=\la  \del^{\bs \nu} (\bs m \times \nabla \bs m), \nabla (\phi^2 \del^{\bs \nu} \bs m)\ra =\la \left( \bs m \times \del^{\bs \nu}  \nabla   \bs m +\bs R_{\bs \nu} \right), \left(  \phi^2 \del^{\bs \nu} \nabla \bs m +
2 \phi \, \nabla \phi \, \del^{\bs \nu} \bs m \right) \ra,
\]
which is bounded by
\[
 \|\phi \, \del^{\bs \nu}  \nabla  \bs m \|_{L^2} \Big(2 \|\nabla \phi  \; \del^{\bs \nu}  \bs m \|_{L^2} + \| \phi \, \bs R_{\bs \nu}\|_{L^2} \Big) + 2  \| \phi \,  \bs R_{\bs \nu}\|_{L^2} \|\nabla \phi  \; \del^\nu  \bs m \|_{L^2}, 
\]
where $|\bs R_{\bs \nu}| \lesssim \sum_{|\ell_1| + |\ell_2| = |\bs \nu|-1} | \nabla^{\ell_1}(\nabla \bs m) \otimes  \nabla^{\ell_2}( \nabla \bs m)|$. Hence for $t \in [-1,0]$ fixed
\[
 \| \phi \, \bs R_{\bs \nu}\|_{L^2} \le  \| \phi \,  \bs R_{\bs \nu}\|_{L^2(B_1)}\le c \|\nabla \bs m\|_{L^\infty(B_1)} \|\nabla \bs m\|_{H^{k}(B_1)}.
\]
In fact, by Sobolev extension (preserving $L^\infty$ bounds) of $\nabla \bs m|_{B_1}$ to a map $\bs g \in L^\infty \cap H^k(\R^2;\R^{6})$
with an equivalent $L^\infty\cap H^k$ bound,
Moser's product estimate applies. Hence for arbitrary $\delta>0$
\[
|I| \le \delta  \|\phi \, \del^\nu  \nabla  \bs m \|_{L^2}^2 + C(\delta) \Big(1+\|\nabla \bs m  \|_{L^\infty(P_1)}^2 \Big)  \|\nabla \bs m(t)\|_{H^{k}(B_{1})}^2
\]
so that the first term can be absorbed for $\delta \lesssim \alpha$.
\end{proof}

\subsection*{Energy estimates}
In proving Theorem \ref{thm:3} we shall argue on the level of energy. We have the following energy inequality 
for regular solutions $\m=\m_\eps$ of \eqref{eq:LLG_MF} on a time interval $[0,T]$.

\begin{lemma}[Energy inequality] \label{lemma:local}
There exists a universal constant $\lambda >0$ such that for $\eps \ge 0$ and $\varphi:\R^2 \to \R$ smooth with compactly supported gradient
\[
\begin{split}
\frac{\alpha}{2} & \int_{0}^{T} \int_{\R^2} |\del_t \m|^2 \varphi^2 dx dt + 
\left[ \int_{\R^2} e_\eps(\m(t)) \varphi^2 \, dx \right]_{t=0}^{T} \\ & \le 
\frac{\lambda}{\alpha} \int_{0}^{T} \int_{\R^2} \Big[ (1+\alpha^2) \nu^2   |\del_z \m(t)|^2 \varphi^2  +  \left
(|\nabla \m|^2 + \eps^2 |\m -\ein_3|^4 \right) |\nabla \varphi|^2 \Big] dx dt.
\end{split}
\]
\end{lemma}

\begin{proof} The claim follows from a standard argument based on the identity 
\[
\alpha |\del_t \m|^2 - \nu (\alpha \del_z \m + \m \times \del_z \m) \cdot \del_t \m = \bs h_{\eps} (\m) \cdot \del_t \m,
\]
where the right hand side produces the time derivative of the density up to a divergence. The corresponding identity 
for the helicity term reads
\[
\left(\nabla \times \m \right) \cdot \del_t \m  = \del_t  \left[ (m_3-1) \nabla \times m \right] - \nabla \times \left[ (m_3-1) \del_t m \right].
\]
Integration by parts and Young's inequality implies the claim.
\end{proof}

If $\varphi \equiv 1$ one can take $\lambda =\frac 1 2$ and obtains in the case $Q(\m)=-1$
\[
\frac{\alpha}{2} \int_0^T \int_{\R^2}  |\del_t \m|^2  \, dx  dt+  \Big[ E_\eps(\m(t) ) \Big]_{t=0}^T  
  \le  \frac{(1+\alpha^2) \nu^2}{4\alpha} \int_0^T  \Big[ D(\m(t)) - 4 \pi \Big]   dt
\]
where we used that
\[
2 \int_{\R^2}  |\del_z \m(t)|^2  \, dx =  D(\m) - 4\pi. 
\]
Lemma \ref{lem:lowerbound} implies for $\eps \le 1/8$ and $E_\eps(\m) < 4 \pi$ that
\[
 D(\m) - 4\pi  <   32 \pi \eps.
\]

\begin{proposition} \label{prop:energy} Suppose $0<\eps \le 1/8$ and $E_\eps(\m(0)) \le 4 \pi -c\eps$, then
%\[
%E_\eps(\m(T)) \le 4 \pi + 8 \pi \eps \left(    \frac{(1+\alpha^2) \nu^2}{ \alpha}   T - \frac{c}{8 \pi} \right)
%\]
%and
%\[
%\frac{\alpha}{2} \int_0^T \int_{\R^2}  |\del_t \m|^2  \, dx  dt  \le 64 \pi \eps \left( \frac 1 2 + \frac{(1+\alpha^2) \nu^2}{\alpha} T \right) 
%  \]
\[
E_\eps(\m(T))< 4\pi \quad \text{for all} \quad 
0<T < \frac{c \alpha}{32 \pi (1+\alpha^2) \nu^2}.
\]
Moreover as $\eps \to 0$
\[
 \sup_{0 \le t \le T}  \int_{\R^2}  |\del_z \m(t)|^2  \, dx = O(\eps) \quad \text{and} \quad   \int_0^T \int_{\R^2}  |\del_t \m|^2  \, dx dt = O(\eps).  
\]

\end{proposition}

Next we show that the energy density $e_\eps(\m(t)): \R^2 \to [0, \infty)$ remains concentrated along the flow. To this end
we invoke Lemma \ref{lemma:local} with $\varphi_R(x)=\varphi(x/R)$, where $\varphi(x)=1$ for $|x| \ge 2$ and $\varphi(x)=0$ for $|x| \le 1$.
By virtue of H\"older's inequality 
%
%and the fact that in the present situation $\mathrm{Excess}(\m) \le 64 \pi \eps$ and 
%\[
%D(\m) + \eps V(\m) \le 2 E_\eps(\m) \le 8 \pi
%\quad
%\text{for} \quad \eps \ll 1,
%\] 
we obtain the estimate
\[
\begin{split}
\frac{\alpha}{2} & \int_{0}^{T} \int_{\R^2} |\del_t \m|^2 \varphi_R^2 dx dt + 
\left[ \int_{\R^2} e_\eps(\m(t)) \varphi_R^2 \, dx \right]_{t=0}^{T} \\ & \le 
c  \int_{0}^{T}  \nu^2  \Big[D(\m) -4\pi \Big] +  R^{-2} E_\eps(\m) \; dt
%\left( \frac{D(\m)}{R^2} + \left( \frac{V(\m)}{R^2} \right)^{2/p} \right) dt  
\end{split}
\]
for generic constants $c$ that only depend on $\alpha$ and $\varphi$ from which we obtain:

\begin{lemma}\label{lemma:decay} There exists a constant $c=c(\alpha)$ such that 
 \[
 \int_{\{|x|>2R\}} e_\eps(\m(t))  \, dx \le \int_{\{|x|>R\}} e_\eps(\m(0)) \, dx + c \, \left( 1+ \eps  (\nu/R)^2 \right)  \, T/R^2
 \]
for all $0 \le t \le T$, $R>0$ and $\eps >0$.
\end{lemma}

\subsection*{Small energy regularity}
The main strategy for proving regularity has been developed in the context of harmonic map heat flows and is well-established  \cite{Struwe:85, Guo_Hong:93, Harpes:04}.
The terminal time $T^\ast$ depends on the initial data and the parameters $\eps$ and $\nu$.
The only possible scenario of finite time blow-up is $|\nabla \m(x_k,t_k)| \to \infty$ for some sequence
$x_k \in \R^2$ and $t_k \nearrow T^\ast$.
We shall show that for moderately small $\eps$, this scenario can be ruled out as long as $E_\eps(\m(t)) < 4 \pi$.

\begin{proposition} \label{prop:regularity} There exists $\eps_0>0$ such that if $0<\eps< \eps_0$ and $E_\eps(\m(t)) < 4 \pi$ 
for all $t <T^\ast$ and  $T^\ast < \infty$, then 
\[
\displaystyle{\limsup_{t \nearrow T^\ast} E_\eps(\m(t)) = 4 \pi.}
\]
\end{proposition}

It is customary to prove small-energy regularity using Schoen's trick, which is well-established for
harmonic maps and flows.

\begin{lemma}\label{lemma:partial}
There exists $\delta_0>0$ such that if 
$\bs m$ is a blow-up solution in $\overline{P_2}$ with
\[
 \int_{B_2(0)} |\nabla \bs m(s)|^2  dy < \delta_0 \quad \text{for all} \quad s \in (-4,0) 
 \]
then
\[ |\nabla \bs m| \le 2 \quad  \text{in} \quad \overline{P_{1}(0)}.\]
\end{lemma}

\begin{proof} There exists $\rho \in [0,2)$ such that 
\[
(1-\rho)^2 \sup_{P_\rho} |\nabla \bs m|^2= \max_{\sigma \in [0,2]} (1-\sigma)^2 \sup_{P_\sigma} |\nabla \bs m|^2.
\]
We set $s_0= |\nabla \bs m(z_0)|=\sup_{P_\rho} |\nabla \bs m|$ for some $z_0 \in \overline{P_\rho(0)}$ and claim
\[
(2-\rho)^2 s_0^2 \le 4.
\] 
Then it follows that $\sup_{P_{1}}|\nabla \bs m|^2 \le (2-\rho)^2 s_0^2 \le 4$, which implies the claim.\\

If otherwise $(2-\rho)^2 s_0^2 > 4$, then in particular $s_0>\frac{2}{2-\rho} \ge 1$. So $\lambda=1/s_0$ is an admissible scaling parameter.
For $(x,t) \in P_1$ we consider the blow-up solution
\[
\tilde \m (x,t)=\bs m(x_0+s_0^{-1} x, t_0+s_0^{-2} t),
\]
for which 
\[
\sup_{P_1}|\nabla \tilde \m|^2 \le s_0^{-2} \sup_{P_{s_0^{-1}}(x_0)}|\nabla \bs m|^2 \le s_0^{-2} \sup_{P_{\frac{1}{2}(2-\rho)} (x_0)}|\nabla \bs m|^2 
\le  s_0^{-2}  \frac{(2-\rho)^2 s_0^2}{\left[\frac 12(2- \rho)\right]^2} \le  4.
\]
Hence it follows from Sobolev embedding $H^2(B_{1/4}) \hookrightarrow L^\infty(B_{1/4})$ and Lemma \ref{lemma:local_Sobolev} applied twice to $\tilde \m$ (being a blow-up solution) that for a generic constant $c$
\[
1=|\nabla \tilde \m(0,0)|^2 \le c \,  \|\nabla \tilde \m(0)\|_{H^2(B_{1/4})}^2 \le c \int_{-1}^0 \|\nabla \tilde \m(t) \|_{L^2(B_1)}^2\, dt. 
\]
But then $
1 \le c \int_{-4}^{0} \|\nabla \bs m(t) \|_{L^2(B_2)}^2\, dt  < 4 c \delta_0$, impossible 
for appropriate $\delta_0>0$.
\end{proof}

%\begin{remark} Invoking Lemma \ref{lemma:local_Sobolev}
%\[
%\int_{P_{R/2}}|\nabla^2 \m|^2 \, dz \le C \, \eps_0
%\]
%\end{remark}
%

\begin{proof}[Proof of Proposition \ref{prop:regularity}] Suppose $T^\ast< \infty$.
It follows from Lemma \ref{lemma:decay} that 
there exist $R_0>0$ and $\eps_0>0$ such that 
\[
\int_{\{ |x|>2R_0\}} |\nabla \m(t)|^2 \, dx < \delta_0 \quad \text{for all} \quad  0<t<T^\ast
\]
if $\eps<\eps_0$ and $\m=\m_\eps$ is a solution with $E_\eps(\m(t)) < 4 \pi$ for all $0 \le t<T^\ast$. Hence for fixed $\eps< \eps_0$, according to
Lemma \ref{lemma:partial},  $|\nabla \m_\eps(x,t)|$ is uniformly bounded for $|x|>3R_0$ and $0<t<T^\ast$.
%The bounds may however depend on $T^\ast$ and hence on $\eps$, but recall that $\eps$ is small but fixed. 
It follows that blow-up can only occur in a finite domain, and it remains
to perform a bubbling analysis as in \cite{Struwe:85}:

\medskip

 Note that by Lemma \ref{lemma:partial} and Proposition \ref{prop:energy}
 the singular set must be finite. Hence after translation and dilation we may assume $\m \in C^\infty(P_2 \setminus \{(0,0)\})$
and claim that if $\eps$ is sufficiently small and $\m$ has a singularity in the origin, then %for 
\[
\limsup_{t \nearrow 0} E_\eps(\m(t);B_2(0)) \ge 4\pi.
\]
If $(0,0)$ is a singularity then by virtue of Lemma \ref{lemma:partial}
\[
\int_{B_{r_k}(x_k)}|\nabla \m(t_k)|^2 dx = \sup_{(x, t) \in B_1\times  (-1,t_k)} \int_{B_{r_k}(x)}|\nabla \m(t)|^2 dy =\frac{\delta_0}{2}
\]
for suitable sequences $x_k \to 0$, $t_k \nearrow 0$ and $r_k \searrow 0$.
The blow-up solution
\[
\bs m_k(x,t)=\m(x_k+r_k x, t_k +r_k^2 t)
\]
defined for $x \in\R^2$ and $-1/r_k^2 \le t  \le 0$ 
%satisfies 
%\[
%\int_{B_{1}}|\nabla \m_k(0)|^2 dx = \int_{B_{r_k}(x_k)}|\nabla \m(t_k)|^2 \, dx 
% = \int_{B_{r_k}(x_k)} e_\eps(\m(t_k)) \, dx + O(r_k)
%\]
%
%and 
solves the perturbed Landau-Lifshitz-Gilbert equation
\[
\del_t \m_k = \m_k \times \left( \alpha \del_t \m_k - \Delta \m_k \right) + \bs f_k
\]
for a field $\bs f_k \perp \m_k$ with
\[
|\bs f_k| \lesssim r_k |\nabla \m_k| + r_k^2 |\m_k-\ein_3|^{3}
\]
hence $\|\bs f_k(t)\|_{L^2} = O(r_k)$ uniformly for all admissible $t$. According to Lemma \ref{lemma:partial} and iterations of Lemma \ref{lemma:local_Sobolev}, 
$\m_k$ satisfies uniform higher order regularity bounds in $P_{1/r_k}$.
It follows from the energy inequality for $\m$ that 
$\int_{-1}^0 \int_{\R^2} |\del_t \m_k|^2 \, dxdt \to 0$ as $k \to \infty$, hence
$\bs v_k=(\del_t \bs m_k )(\tau_k)$ and $\bs w_k = \bs f_k(\tau_k)$ converge to zero in $L^2(\R^2)$
 for some sequence $\tau_k \nearrow 0$.
Note that $\bs u_k = \bs m_k(\tau_k)$ is an almost harmonic map in the sense that
\[ 
 \bs u_k \times \Delta \bs u_k = \alpha \, \bs u_k \times\bs v_k - \bs v_k +\bs w_k
\]
and subconvergence strongly in $H^1_{\rm loc}(\R^2)$ to a harmonic map $\bs u$ of finite energy in $\R^2$.  To show that $\bs u$ is non-constant 
we invoke the local energy equality for $\m_k$
\[
\int_{B_1} |\nabla \m_k(0)|^2 \, dx -  \int_{B_{2}} |\nabla \m_k(\tau_k )|^2 \, dx  \le c \int_{\tau_k}^0   \int_{B_2} \left(  |\nabla \m_k|^2+  |\bs f_k|^2 \right)  \,dx dt = O(\tau_k),
\]
which implies that 
\[
 \int_{B_2} |\nabla \bs u_k|^2 \, dx =\int_{B_{2}}|\nabla \m_k(\tau_k)|^2 dx \ge \frac{\delta}{2}+ O(\tau_k).
 \]
 By strong convergence $\int_{B_{2}} |\nabla \bs u|^2 \, dx >0$, and by virtue of well-known theory about harmonic maps $\frac 1 2 \int_{\R^{2}} |\nabla \bs u|^2 \, dx = 4 \pi$.
The rescaled energy densities
\[
e_{\eps, k}(\bs u):=  \frac{|\nabla \bs u|^2}{2} +  \eps  r_k \left( (\bs u -\ein_3) \cdot (\nabla \times \bs u) +  \frac{r_k}{16} |\bs u -\ein_3|^4    \right)
\]
are non-negative for $\eps$ sufficiently small, independently of $k$. Hence by letting $s_k = t_k+r^2_k \tau_k \to 0$ we have for arbitrary $R_0>0$
\[
\int_{B_2(0)} e_{\eps}(\m(s_k)) \, dx \ge \int_{B_1(x_k)} e_{\eps}(\m(s_k)) \, dx = \int_{B_{1/r_k}} e_{\eps,k}(\bs u_k) \, dx \ge   \int_{B_{R_0}} e_{\eps,k}(\bs u_k) \, dx 
\]
for $k > k_0$ depending on $R_0$, and  $\int_{B_{R_0}} e_{\eps,k}(\bs u_k) \, dx = \frac{1}{2}  \int_{B_{R_0}} |\nabla \bs u_k|^2 \, dx + O(r_k)$ as $k \to \infty$
which implies the claim by strong $L^2(B_{R_0})$ compactness of $\nabla \bs u_k$.
%
%\[
% \int_{B_{R}} e_\eps(\m(t))  \, dx \le \int_{B_{2R}} e_\eps(\m(0)) \, dx +\frac{CT}{R^2}
% \]
% for a constant $C=C(\alpha, \nu,  \eps)$ if $0<R<1$, and  
%
\end{proof}

\begin{proof}[Proof of Theorem \ref{thm:3}] The first claim has been discussed in the forefront of the theorem.
The second follows from Proposition \ref{prop:energy} and Proposition \ref{prop:regularity}.
For the third claim we deduce  from Lemma \ref{lem:lowerbound} as in the proof of Theorem~\ref{thm:2} that $\limsup_{\eps \to 0} V(\m_{\eps}^0)<\infty$ 
and $\lim_{\eps \to 0} D(\m_{\eps}^0)= 4 \pi$, hence $\m_0 \in \mathcal{C}$.
Moreover, it follows from Proposition \ref{prop:energy} that for every sequence $\eps_k \searrow 0$ the corresponding solutions
$\m_{\eps_k}$ subconverge weakly to a weak solution of $\m$ of the standard Landau-Lifshitz-Gilbert equation
\[
\del_t \m = \alpha \, \m \times \del_t \m - \nabla \cdot \left( \m \times \nabla \m\right) \quad \text{with} \quad \m(0)=\m_0.
\] 
Since $\del_t \m=0$ by Proposition \ref{prop:energy}, it follows that $\m \equiv \m_0$. Now for every $t \in [0,T]$ the sequence
$\nabla \m_{\eps_k}(t)$ converges weakly to $\nabla \m_0$ with $\lim_{k \to \infty}D( \m_{\eps_k}(t))=4\pi$, which implies strong convergence.
Finally we deduce convergence of the whole family as $\eps \searrow 0$.
\end{proof}

\appendix
\section{Cut-off lemma} 

The following cut-off result in the spirit of \cite{esteban,Lin_Yang, CSK} is crucial for the proof of Proposition~\ref{prop:compactness}:
\begin{lemma}\label{lem:cutoff}
Suppose $\m\colon \R^2\to \St$ satisfies $\int_{\R^2} \lvert \nabla \m \rvert^2 \, dx < \infty$ and
\[ 
  \int_{B_{4R}\setminus B_R} \lvert \nabla \m \rvert^2 \, dx + \sigma \int_{B_{4R}\setminus B_R} \tfrac{1}{2^p} \lvert \m-\ein_3 \rvert^p \, dx  < \delta
\]
for some $0<\delta\ll 1$, $R \geq 1$, $\sigma\in\{0,1\}$. Then, there exist
\[ \m^{(1)}, \m^{(2)} \colon \R^2\to \St \quad \text{with} \quad \int_{\R^2} \lvert \nabla \m^{(i)} \rvert^2 \, dx < \infty \quad\text{for}\quad i=1,2, \]
some $c\in [R,2R]$ and a constant $C=C(\delta,R)<\infty$ so that
\begin{gather*}
\m^{(1)} = \m \quad \text{on }B_c, \qquad V(\m^{(1)}) \lesssim C,\\
\int_{\R^2\setminus B_c} \lvert \nabla \m^{(1)} \rvert^2 \, dx + \sigma \int_{\R^2\setminus B_c} \tfrac{1}{2^p} \lvert \m^{(1)}-\ein_3 \rvert^p \, dx \lesssim \delta + \sigma (\tfrac{\delta}{R^2})^{2/p},
\end{gather*}
and
\begin{gather*}
\m^{(2)} = \m \quad \text{on }\R^2\setminus B_{2c}, \qquad  V(\m^{(2)}) \lesssim V(\m)+C,\\
\int_{B_{2c}} \lvert \nabla \m^{(2)} \rvert^2 \, dx + \sigma \int_{B_{2c}} \tfrac{1}{2^p} \lvert \m^{(2)}-\ein_3 \rvert^p \, dx \lesssim \delta + \sigma (\tfrac{\delta}{R^2})^{2/p}.
\end{gather*}
%where $L$ depends on $R$, $\sigma$ and $\delta$ only. In particular, we have
%\[ V(\tilde{m}) = \int_{\R^2} \tfrac{1}{2^p} \lvert \tilde{\m} - \ein_3 \rvert^p \, dx \lesssim (R+L)^2. \]
%{\color{red} If, in addition,
%\[ \int_{B_{2R}\setminus B_R} \tfrac{1}{2^p} (1-m_3)^p dx < \delta^p, \]
%we may assert
%\[ \int_{B_{2R}\setminus B_R} \tfrac{1}{2^p} (1-\tilde{m}_3)^p dx \lesssim  \delta^p. \]}
\end{lemma}

\begin{proof}
We proceed in several steps. The symbol $\lesssim$ will denote an inequality that holds up to a generic, universal multiplicative constant that may change from line to line.

\medskip
\noindent \textbf{Step~1} (Choice of radius $c$): We consider $\m$ in polar coordinates and write $\m(x)=\m(r,\theta)$. Moreover, we define
\[ g\colon [R,4R] \to \R, \quad g(r):= \int_0^{2\pi} \bigl( \lvert \partial_r \m \rvert^2 + \lvert \tfrac{1}{r} \partial_\theta \m \rvert^2 + \sigma \underbrace{\tfrac{1}{2^p}\lvert \m-\ein_3 \rvert^p}_{=2^{-\frac{p}{2}}(1-m_3)^{\frac{p}{2}}} \bigr) d\theta. \]
Poincar\'e's inequality yields
\[ 
  \lVert \m(r,\cdot) - \bar{\m}(r) \rVert_\infty^2 \lesssim \int_0^{2\pi} \lvert \partial_\theta \m(r,\theta) \rvert^2 d\theta \quad \forall r>0,
\]
where $\bar{\m}(r):=\dashint_0^{2\pi} \m(r,\theta) \, d\theta$. 
Hence, we may choose $c\in[R,2R]$ so that
\begin{align*}
\tfrac{\delta}{R} &\geq \tfrac{1}{R}\int_{B_{4R}\setminus B_R} \lvert \nabla \m \rvert^2 + \sigma \tfrac{1}{2^p}\lvert \m-\ein_3 \rvert^p \, dx = \tfrac{1}{R} \int_R^{4R} g(r) \, r \, dr\\
&\gtrsim\dashint_{R}^{2R} \bigl( g(r)+g(2r) \bigr) r \, dr  \geq \bigl( g(c)+g(2c)\bigr) c.
\end{align*}
By definition of $g$, we obtain
\begin{align*}
1-\lvert \bar{\m}(r) \rvert^2 &= \dashint_{0}^{2\pi} \lvert \m(r,\theta)-\bar{\m}(r) \rvert^2 \, d\theta \lesssim \lVert \m(r,\cdot) - \bar{\m}(r) \rVert_\infty^2\\
&\lesssim \int_0^{2\pi} \lvert \partial_\theta \m(r,\theta) \rvert^2 d\theta \lesssim R g(r)r \lesssim \delta \quad \text{for }r=c,2c,
\end{align*}
%&\geq \int_0^{2\pi} \bigl( \lvert \partial_r \m(c,\theta) \rvert^2 + \lvert \tfrac{1}{c} \partial_\theta \m(c,\theta) \rvert^2  + \tfrac{\sigma}{2^p}\bigl(1-m_3(c,\theta)\bigr)^p\bigr) c \, d\theta\\
%&\gtrsim \tfrac{1}{c} \int_0^{2\pi} \lvert \partial_\theta \m \rvert^2 \, d\theta + c\sigma \Bigl(\int_0^{2\pi} \bigl( 1-m_3(c,\theta) \bigr) d\theta \Bigr)^p\\
%&\gtrsim \tfrac{1}{c} \lVert \m(c,\cdot) - \bar{\m}(c) \rVert_\infty^2 + c\sigma \bigl(1-\bar{m}_3(c)\bigr)^p,
%\end{align*}
and
\begin{align*}
\sigma \bigl(1-\bar{m}_3(r)\bigr) &= \sigma^2 \dashint_0^{2\pi} \bigl(1-m_3(r,\theta)\bigr) \, d\theta \lesssim \sigma\Bigl( \sigma \int_0^{2\pi} \bigl(1-m_3(r,\theta)\bigr)^{\frac{p}{2}} \, d\theta \Bigr)^{\frac{2}{p}}\\
&\lesssim \sigma \bigl( g(r) \bigr)^{\frac{2}{p}} \lesssim \sigma (\tfrac{\delta}{R^2})^{\frac{2}{p}} \quad \text{for }r=c,2c.
\end{align*}
In particular, we may assume $\lvert \bar{\m}(c) \rvert \geq \tfrac{1}{2}$.

\medskip
\noindent \textbf{Step~2} (Definition of $\m^{(1)}$):

Let
\[
\e:= 
\left.
\begin{cases}
\tfrac{\bar{\m}(c)}{\lvert \bar{\m}(c) \rvert}, & \sigma=0\\
\ein_3, &\sigma=1
\end{cases}
\right\}\in\St,
\]
so that for $\sigma=0$ we have
\[
\lVert \m(c)-\e \rVert_\infty^2 
\lesssim \underbrace{\lVert \m(c,\cdot)-\bar{\m}(c) \rVert_\infty^2}_{\lesssim \delta} + \underbrace{\lvert \bar{\m}(c)-\e \rvert^2}_{= (1-\lvert \bar{\m}(c) \rvert )^2 \lesssim \delta^2} \lesssim \delta. 
\]
If $\sigma=1$, we may modify the second estimate as follows:
\begin{align*}
\lvert \bar{\m}(c)-\ein_3 \rvert^2 \leq \dashint_0^{2\pi} \underbrace{\lvert \m(c,\theta) - \ein_3 \rvert^2}_{=2(1-m_3(c,\theta))} \, d\theta \lesssim (1-\bar{m}_3(c)) \lesssim (\tfrac{\delta}{R^2})^{2/p}.
\end{align*}
Hence, in either situation,
\[
\lVert \m(c,\cdot)- \e \rVert_\infty^2  \lesssim \delta + \sigma  (\tfrac{\delta}{R^2})^{2/p} \ll 1.
\]
We will define $\m^{(1)} \colon \R^2 \to \St$ in two steps:

\medskip
\noindent \textbf{Step~2a} (Definition of $\m^{(1)}$ on $B_{2c}$):
Let $\eta\colon \R\to[0,1]$ be a smooth cut-off function with $\eta(s)=1$ for $s\leq 0$ and $\eta(s)=0$ for $s\geq 1$. We define
\[ 
\m^{(1)}(r,\theta) =
\begin{cases}
\frac{\eta(\frac{r-c}{c})\m(c,\theta)+(1-\eta(\frac{r-c}{c}))\e}{\lvert \eta(\frac{r-c}{c})\m(c,\theta)+(1-\eta(\frac{r-c}{c}))\e\rvert}, & c<r<2c,\\
\m(r,\theta), & 0\leq r \leq c.
\end{cases}
\]
so that $\m^{(1)}$ has a well-defined trace across $\partial B_c$.
Using the inequality
\[
\lvert \partial_i (\rho \m^{(1)}) \rvert^2 = \rho^2 \lvert \partial_i \m^{(1)} \rvert^2 + \lvert \partial_i \rho \rvert^2 \geq \tfrac{1}{4} \lvert \partial_i \m^{(1)} \rvert^2, \quad i=r,\theta,
\]
where
\[
\rho=\bigl\lvert \eta(\tfrac{r-c}{c})\m(c,\theta) + \bigl(1-\eta(\tfrac{r-c}{c})\bigr) \e \bigr\rvert \geq \tfrac{1}{2},
\]
we obtain for $c\leq r \leq 2c$ that
\[ \lvert \partial_r \m^{(1)}(r,\theta) \rvert^2 \lesssim \lvert \tfrac{1}{c} \eta'(\tfrac{r-c}{c})(\m(c,\theta)-\e) \rvert^2 \lesssim \tfrac{1}{c^2} \lVert \m(c,\cdot)-\e \rVert_\infty^2 \lesssim \tfrac{\delta+(\delta R^{-2})^{\frac{2}{p}}}{r^2}.  \]
and
\[ \lvert \tfrac{1}{r} \partial_\theta \m^{(1)}(r,\theta) \rvert^2 \lesssim \bigl\lvert \tfrac{1}{r} \partial_\theta \m(c,\theta) \eta(\tfrac{r-c}{c}) \bigr\rvert^2 \lesssim \tfrac{1}{r^2} \lvert \partial_\theta \m(c,\theta) \rvert^2.  \]
Hence,
\begin{align*}
\int_c^{2c} \int_0^{2\pi} &\bigl( \lvert \partial_r \m^{(1)}(r,\theta) \rvert^2 + \lvert \tfrac{1}{r} \partial_\theta \m^{(1)} \rvert^2 \bigr) d\theta\, r\,dr \\
&\lesssim \int_c^{2c} \int_0^{2\pi} \bigl( \tfrac{\delta+(\delta R^{-2})^{\frac{2}{p}}}{r^2} + \tfrac{1}{r^2} \lvert \partial_\theta \m(c,\theta) \rvert^2 \bigr) d\theta \, r \, dr\\
&\lesssim \underbrace{\int_c^{2c} \tfrac{dr}{r}}_{=\ln 2} \; \Bigl(\delta + (\tfrac{\delta}{R^2})^\frac{2}{p} + \underbrace{\int_0^{2\pi} \lvert \partial_\theta \m(c,\theta) \rvert^2 d\theta}_{\lesssim \delta}\Bigr) \lesssim \delta + (\tfrac{\delta}{R^2})^\frac{2}{p}.
\end{align*}
Finally, since
\begin{align*}
1=\lvert \m \rvert = \bigl\lvert (1-\eta)\m + \eta \e + \eta (\m-\e) \bigr\rvert
\leq \bigl\lvert (1-\eta)\m + \eta \e \bigr\rvert +  \lvert \m-\e \rvert
\end{align*}
implies
\begin{align*}
1-\bigl\lvert (1-\eta) \m + \eta \e \bigr\rvert \leq \lvert \m - \e \rvert,
\end{align*}
we obtain for $\rho$ as above
\begin{align*}
\lvert \m^{(1)} - \e \rvert \leq \underbrace{\bigl\lvert \m^{(1)} - \rho \m^{(1)}\rvert}_{=1-\rho \leq \lvert\m(c,\theta)-\e \rvert} + \underbrace{\bigl\lvert \bigl(\eta \e+(1-\eta) \m(c,\theta)\bigr) - \e \bigr\rvert}_{\lesssim \lvert \m(c,\theta)-\e \rvert} \lesssim \lvert \m(c,\theta) - \e \rvert.
\end{align*}
Hence, in the case $\sigma=1$
\begin{align*}
\MoveEqLeft\int_c^{2c} \int_0^{2\pi} \tfrac{1}{2^p}\lvert \m^{(1)}-\ein_3 \rvert^{p} \, d\theta \,r \,dr
\lesssim \int_c^{2c} \int_0^{2\pi} \lvert \m(c,\theta)- \ein_3 \rvert^{p} \, d\theta \, r\, dr\\
&\lesssim \int_c^{2c} \underbrace{\int_0^{2\pi}\bigl(1-m_3(c,\theta)\bigr)^{\frac{p}{2}} \, d\theta}_{\lesssim \frac{\delta}{R^2}} c \, dr \lesssim \delta.
\end{align*}
Therefore, we have
\[ \int_{B_{2c}\setminus B_c} \lvert \nabla \m^{(1)} \rvert^2 \,dx + \sigma \int_{B_{2c}\setminus B_c} \tfrac{1}{2^p} \lvert \m^{(1)}-\ein_3 \rvert^p \, dx \lesssim \delta + \sigma \bigl(\tfrac{\delta}{R^2}\bigr)^\frac{2}{p}  \]

\medskip
\noindent \textbf{Step~2b} (Definition of $\m^{(1)}$ on $\R^2\setminus B_{2c}$):
If $\sigma=1$, there is nothing left to be done and we may just set $\m^{(1)}\equiv \ein_3$ on $\R^2\setminus B_{2c}$. Otherwise, we will define $\m^{(1)}$ on $(2c,2c+L)$ for some $L\gg 2c$ (to be chosen later) by interpolating $\e$ with $\ein_3$. Indeed, let $\gamma\colon [0,1]\to \St$ denote a smooth curve that connects $\gamma(0)=\e$ with $\gamma(1)=\ein_3$. Assume w.l.o.g. that $\lvert \tfrac{d}{ds}\gamma(s) \rvert \lesssim 1$ independently of $\e\in \St$.
We introduce a logarithmic cut-off function 
\[ \eta_L \colon [2c,2c+L]\to[0,1], \quad \eta_L(r):=\tfrac{\ln(\frac{r}{2c})}{\ln(\frac{2c+L}{2c})}, \]
and let
\[ \m^{(1)}(r,\theta) =
\begin{cases}
\gamma(\eta_L(r)), &2c\leq r \leq 2c+L\\
\ein_3 , &2c+L< r.
\end{cases}
\]
Then, $\m^{(1)}$ has a well-defined trace both across $\partial B_{2c}$ and $\partial B_{2c+L}$, and
\[ \tfrac{d}{dr} \m^{(1)}(r) = \frac{(\tfrac{d}{ds}\gamma)(\eta_L(r))}{r\ln(\frac{2c+L}{2c})}. \]
Hence, $\partial_\theta \m^{(1)}=0$ and
\begin{align*}
\int_{2c}^{2c+L} \int_0^{2\pi} \underbrace{\lvert \partial_r \m^{(1)}(r,\theta) \rvert^2}_{\lesssim \frac{1}{r^2} \ln^{-2}(\frac{2c+L}{2c})} d\theta\,r\,dr \lesssim \tfrac{1}{\ln^2(1+\frac{L}{2c})} \int_{2c}^{2c+L} \tfrac{dr}{r} = \tfrac{1}{\ln(1+\frac{L}{2c})} \lesssim \delta,
\end{align*}
if $L = 2c(e^{\frac{1}{\delta}}-1)$.

Thus, we may conclude for $\sigma\in\{0,1\}$:
\[ \int_{\R^2\setminus B_{2c}} \lvert \nabla \m^{(1)} \rvert^2 \, dx + \sigma \int_{\R^2\setminus B_{2c}} \tfrac{1}{2^p} \lvert \m^{(1)} - \ein_3 \rvert^p \, dx \lesssim \delta+ \sigma \bigl(\tfrac{\delta}{R^2}\bigr)^{\frac{2}{p}}, \]
and
\[ V(\m^{(1)}) = \int_{B_{2c+L}} \underbrace{\tfrac{1}{2^p}\lvert \m^{(1)}-\ein_3 \rvert^p}_{\leq 1} \, dx \lesssim (2c+L)^2 =: C(\delta,R).  \]

\medskip
\noindent \textbf{Step~3} (Definition of $\m^{(2)}$): In order to define $\m^{(2)}$, we proceed as in Step~2. Let
\[ \e := \tfrac{\bar{\m}(2c)}{\lvert \bar{\m}(2c) \rvert} \in \St. \]
Then
\[ \lVert \m(2c,\cdot)-\e \rVert^2_\infty \lesssim \delta + \sigma \bigl( \tfrac{\delta}{R^2} \bigr)^{\frac{2}{p}} \ll 1, \]
and, using the same cut-off function $\eta\colon \R\to[0,1]$ as before, we may define $\m^{(2)} \colon \R^2 \to \St$ as
\begin{align*}
\m^{(2)}(r,\theta) := 
\begin{cases}
  \e, & r\leq c,\\
  \frac{\eta(\frac{r-c}{c})\e+(1-\eta(\frac{r-c}{c}))\m(2c,\theta)}{\lvert \eta(\frac{r-c}{c})\e+(1-\eta(\frac{r-c}{c}))\m(2c,\theta) \rvert}, &c<r<2c,\\
  \m(r,\theta), &r\geq 2c,
\end{cases}
\end{align*}
so that $\m^{(2)}$ has a well-defined trace across $\partial B_c$ and $\partial B_{2c}$.% $\m^{(2)}(2c,\theta)=\m(2c,\theta)$ and $\tilde{\m}(c,\theta)=\e$.

As before, we estimate for $c< r <2c$
\[ \lvert \partial_r \m^{(2)}(r,\theta) \rvert^2 \lesssim \tfrac{\delta+(\delta R^{-2})^\frac{2}{p}}{r^2} \quad\text{and}\quad \lvert \tfrac{1}{r} \partial_\theta \m^{(2)}(r,\theta) \rvert^2  \lesssim \tfrac{1}{r^2} \lvert \partial_\theta \m(2c,\theta) \rvert^2,  \]
so that
\begin{align*}
\int_c^{2c} \int_0^{2\pi} &\bigl( \lvert \partial_r \m^{(2)}(r,\theta) \rvert^2 + \lvert \tfrac{1}{r} \partial_\theta \m^{(2)} \rvert^2 \bigr) d\theta\, r\,dr \lesssim \delta+\bigl( \tfrac{\delta}{R^2}\bigr)^{\frac{2}{p}}.
\end{align*}
Moreover, by the same argument as in Step~2, for $\sigma= 1$:
\begin{align*}
\MoveEqLeft \int_c^{2c} \int_0^{2\pi} \tfrac{1}{2^p}\lvert \m^{(2)}-\ein_3 \rvert^{p} \, d\theta \,r \,dr \lesssim \delta.
\end{align*}
Hence, we may conclude for $\sigma\in\{0,1\}$:
\[ \int_{B_{2c}} \lvert \nabla \m^{(2)} \rvert^2 \, dx + \sigma \int_{B_{2c}} \tfrac{1}{2^p} \lvert \m^{(2)}-\ein_3 \rvert^p \, dx \lesssim \delta + \sigma \bigl( \tfrac{\delta}{R^2}\bigr)^{\frac{2}{p}}, \]
and
\[ V(\m^{(2)}) = \underbrace{\int_{\R^2\setminus B_{2c}} \tfrac{1}{2^p} \lvert \m-\ein_3 \rvert^p \, dx}_{\leq V(\m)} + \int_{B_{2c}} \underbrace{\tfrac{1}{2^p} \lvert \m^{(2)}-\ein_3 \rvert^p}_{\leq 1} \, dx \lesssim V(\m) + \underbrace{(2c)^2}_{=: C(\delta,R)}. \]
\end{proof}

\section{Construction of a stream function}
\begin{lemma}
Given $R>1$, there exists a smooth function $f_R\colon [0,\infty)\to\R$ so that
\[ f_R(r)= \begin{cases} \ln(1+r^2), & \text{for }0\leq r \leq R,\\ \text{const.}, &\text{for }r\geq 2R, \end{cases} \]
and
\[ 0\leq f_R'(r) \leq \tfrac{2r}{1+r^2}, \quad 0\leq -f_R''(r) \leq \tfrac{C}{1+r^2} \quad \text{for all }r\geq R.  \]
\end{lemma}

\begin{proof}
Let $h\colon [0,\infty)\to\R$ be given by (in fact, $h$ is a regularization of the function $y\mapsto \min(y,0)$)
\[ h(y) = \int_0^y \eta(s) \,ds, \]
where $\eta\colon \R\to[0,1]$ is a smooth, non-increasing function with
\[ \eta(s)=1 \quad \text{for }s\leq 0, \qquad \eta(s)=0 \quad \text{for }s\geq \tfrac{1}{2}, \quad 0\leq -\eta'(s) \leq C \quad \forall s\in\R.  \]
Then,
\[ f_R(r):=h\left(\ln(1+r^2)-\ln(1+R^2)\right)+\ln(1+R^2), \quad r\geq 0, \]
satisfies the claim.

Indeed, we have $h(y)=y$ for $y\leq 0$ and $h(y)=\int_0^\infty\eta(s)\,ds$ for $y\geq \tfrac{1}{2}$. Since $\ln(1+r^2)- \ln(1+R^2)\leq 0$ for $r\leq R$, we therefore obtain $f_R(r)= \ln(1+r^2)$. On the other hand, $r\geq 2R\geq 2$ yields $\ln(1+r^2)-\ln(1+R^2) \geq \ln(\frac{1+4R^2}{1+R^2})\geq \ln(\frac{5}{2})\geq\frac{1}{2}$, so that $f_R(r)=\int_0^{\infty} \eta(s) \, ds$.

Finally, we have
\[ f_R'(r)=\underbrace{\eta\bigl(\ln(1+r^2)-\ln(1+R^2)\bigr)}_{\in [0,1]} \tfrac{2r}{1+r^2} \]
and
\[ f_R''(r) = \underbrace{\eta'\bigl(\ln(1+r^2)-\ln(1+R^2)\bigr)}_{\leq 0} \bigl(\tfrac{2r}{1+r^2}\bigr)^2 + \underbrace{\eta\bigl(\ln(1+r^2)-\ln(1+R^2)\bigr)}_{\in[0,1]} \tfrac{2(1-r^2)}{(1+r^2)^2}. \]
In particular, $0\leq f_R'(r) \leq \tfrac{2r}{1+r^2}$ for $r\geq R$ and $0\leq -f_R''(r)\leq \tfrac{C}{1+r^2}$.
\end{proof}

\section{Pulled back Landau-Lifshitz-Gilbert equation} \label{ap:LLG}
We shall argue on the level of the Landau-Lifshitz form 
\[
\begin{split}
 (1+\alpha^2) \del_t \m &+(1+\alpha \beta) (v \cdot \nabla)\m  \\ &= - \left[ (\alpha -\beta) \m \times  (v \cdot \nabla)\m
+ \m \times \bs{h}_{\rm eff} + \alpha \, \m \times \m \times \bs{h}_{\rm eff} \right],
\end{split}
\]
see e.g. \cite{Melcher_Ptashnyk}, rather than the Gilbert form \eqref{eq:LLG_STT}. Solving Thiele's equation we have
\[
(1+\alpha^2) c=(1+\alpha \beta)v-(\alpha - \beta)v^\perp.
\]
Now we compute
\begin{eqnarray*}
(1+\alpha^2) \frac{\mathrm{d}}{\mathrm{d}t} \m(x+ct,t) &=& (1+\alpha^2)  \del_t \m + (1+\alpha^2)  (c \cdot \nabla) \m \\
 &=& (1+\alpha^2)  \del_t \m + (1+\alpha \beta) (v  \cdot \nabla)  \m - (\alpha -\beta) (v \times \nabla) \m \\
  &=& - (\alpha -\beta) \bs{\Psi}    -( \m \times \bs{h}_{\rm eff} + \alpha \, \m \times \m \times \bs{h}_{\rm eff}),
  \end{eqnarray*}
  where with the notation $v \times \nabla=v_1 \del_2 - v_2 \del_1$
\begin{eqnarray*}
\bs{\Psi} &=&  (v \times \nabla) \m  + \m \times  (v \cdot \nabla)\m \\
&=& v_1 \left(  \, \del_2 \m +  \m \times \del_1 \m \right) 
 - v_2 \left(  \,  \del_1 \m - \m \times \del_2 \m \right)\\
&=& 2v_1  \m \times \del_z \m  
 - 2v_2 \del_z \m.  
 \end{eqnarray*}
where $\del_z \m = \frac{1}{2} \left( \del_1 \m -\m \times \del_2 \m \right)$.
Upon the transformation $\m(x+ct,t) \mapsto \m(x,t)$ and with effective coupling parameters 
$\displaystyle{
\nu_i= \frac{2 (\alpha-\beta) v_i}{1+\alpha^2}}$
this can be written as
\[
 (1+\alpha^2) \left( \del_t \m +  \nu_1  \m \times \del_z \m  
 - \nu_2 \del_z \m  \right)  +
 \m \times \bs{h}_{\rm eff} + \alpha \, \m \times \m \times \bs{h}_{\rm eff}
=0.
\]
A rigid rotation yields for $\nu=\sqrt{\nu_1^2+\nu_2^2}$ 
\[
 (1+\alpha^2) \left( \del_t \m -  \nu   \del_z \m   \right)  +
 \m \times \bs{h}_{\rm eff} + \alpha \, \m \times \m \times \bs{h}_{\rm eff}
=0,
\]
which easily recasts into \eqref{eq:LLG_MF}.

\bibliography{references}

\def\cprime{$'$}
\begin{thebibliography}{10}

\bibitem{Alouges_Soyeur:92}
Fran{\c{c}}ois Alouges and Alain Soyeur.
\newblock On global weak solutions for {L}andau-{L}ifshitz equations: existence
  and nonuniqueness.
\newblock {\em Nonlinear Anal.}, 18(11):1071--1084, 1992.

\bibitem{Arthur}
K.~Arthur, G.~Roche, D.~H. Tchrakian, and Yisong Yang.
\newblock Skyrme models with self-dual limits: d=2,3.
\newblock {\em Journal of Mathematical Physics}, 37(6):2569--2584, 1996.

\bibitem{Bogdanov_Hubert:1994}
A.~Bogdanov and A.~Hubert.
\newblock Thermodynamically stable magnetic vortex states in magnetic crystals.
\newblock {\em Journal of Magnetism and Magnetic Materials}, 138(3):255 -- 269,
  1994.

\bibitem{Bogdanov_Hubert:1999}
A.~Bogdanov and A.~Hubert.
\newblock The stability of vortex-like structures in uniaxial ferromagnets.
\newblock {\em Journal of Magnetism and Magnetic Materials}, 195(1):182 -- 192,
  1999.

\bibitem{Brezis_Coron_How_To_Bubbles}
H.~Brezis and J.-M. Coron.
\newblock Convergence of solutions of {$H$}-systems or how to blow bubbles.
\newblock {\em Arch. Rational Mech. Anal.}, 89(1):21--56, 1985.

\bibitem{Brezis_Coron}
Ha{\"{\i}}m Brezis and Jean-Michel Coron.
\newblock Large solutions for harmonic maps in two dimensions.
\newblock {\em Comm. Math. Phys.}, 92(2):203--215, 1983.

\bibitem{Carbou_Fabrie_R3:01}
Gilles Carbou and Pierre Fabrie.
\newblock Regular solutions for {L}andau-{L}ifschitz equation in {$\Bbb R^3$}.
\newblock {\em Commun. Appl. Anal.}, 5(1):17--30, 2001.

\bibitem{Cote_Ignat_Miot}
Rapha{\"e}l C{\^o}te, Radu Ignat, and Evelyne Miot.
\newblock A thin-film limit in the {L}andau-{L}ifshitz-{G}ilbert equation
  relevant for the formation of {N}\'eel walls.
\newblock {\em J. Fixed Point Theory Appl.}, 15(1):241--272, 2014.

\bibitem{esteban}
Maria~J. Esteban.
\newblock A direct variational approach to {S}kyrme's model for meson fields.
\newblock {\em Comm. Math. Phys.}, 105(4):571--591, 1986.

\bibitem{Guo_Hong:93}
Bo~Ling Guo and Min~Chun Hong.
\newblock The {L}andau-{L}ifshitz equation of the ferromagnetic spin chain and
  harmonic maps.
\newblock {\em Calc. Var. Partial Differential Equations}, 1(3):311--334, 1993.

\bibitem{han2010}
Jung~Hoon Han, Jiadong Zang, Zhihua Yang, Jin-Hong Park, and Naoto Nagaosa.
\newblock Skyrmion lattice in a two-dimensional chiral magnet.
\newblock {\em Physical Review B}, 82(9):094429, 2010.

\bibitem{Harpes:04}
Paul Harpes.
\newblock Uniqueness and bubbling of the 2-dimensional {L}andau-{L}ifshitz
  flow.
\newblock {\em Calc. Var. Partial Differential Equations}, 20(2):213--229,
  2004.

\bibitem{Helein_book}
Fr{\'e}d{\'e}ric H{\'e}lein.
\newblock {\em Harmonic maps, conservation laws and moving frames}, volume 150
  of {\em Cambridge Tracts in Mathematics}.
\newblock Cambridge University Press, Cambridge, second edition, 2002.
\newblock Translated from the 1996 French original, With a foreword by James
  Eells.

\bibitem{Komineas2015}
Stavros Komineas and Nikos Papanicolaou.
\newblock Skyrmion dynamics in chiral ferromagnets under spin-transfer torque.
\newblock {\em Phys. Rev. B}, 92:174405, Nov 2015.

\bibitem{KMM_spin}
Matthias Kurzke, Christof Melcher, and Roger Moser.
\newblock Vortex motion for the {L}andau-{L}ifshitz-{G}ilbert equation with
  spin-transfer torque.
\newblock {\em SIAM J. Math. Anal.}, 43(3):1099--1121, 2011.

\bibitem{2D_skyrmion}
Jiayu Li and Xiangrong Zhu.
\newblock Existence of 2{D} skyrmions.
\newblock {\em Math. Z.}, 268(1-2):305--315, 2011.

\bibitem{Lin_Yang}
Fanghua Lin and Yisong Yang.
\newblock Existence of two-dimensional skyrmions via the
  concentration-compactness method.
\newblock {\em Comm. Pure Appl. Math.}, 57(10):1332--1351, 2004.

\bibitem{Lin_Yang_splitting}
Fanghua Lin and Yisong Yang.
\newblock Energy splitting, substantial inequality, and minimization for the
  {F}addeev and {S}kyrme models.
\newblock {\em Comm. Math. Phys.}, 269(1):137--152, 2007.

\bibitem{Lions_CC_L1}
P.-L. Lions.
\newblock The concentration-compactness principle in the calculus of
  variations. {T}he locally compact case. {I}.
\newblock {\em Ann. Inst. H. Poincar\'e Anal. Non Lin\'eaire}, 1(2):109--145,
  1984.

\bibitem{Melcher:11}
Christof Melcher.
\newblock Global solvability of the {C}auchy problem for the
  {L}andau-{L}ifshitz-{G}ilbert equation in higher dimensions.
\newblock {\em Indiana Univ. Math. J.}, 61(3):1175--1200, 2012.

\bibitem{CSK}
Christof Melcher.
\newblock Chiral skyrmions in the plane.
\newblock {\em Proc. R. Soc. Lond. Ser. A Math. Phys. Eng. Sci.},
  470(2172):20140394, 17, 2014.

\bibitem{Melcher_Ptashnyk}
Christof Melcher and Mariya Ptashnyk.
\newblock Landau-{L}ifshitz-{S}lonczewski equations: global weak and classical
  solutions.
\newblock {\em SIAM J. Math. Anal.}, 45(1):407--429, 2013.

\bibitem{Moser_book}
Roger Moser.
\newblock {\em Partial regularity for harmonic maps and related problems}.
\newblock World Scientific Publishing Co. Pte. Ltd., Hackensack, NJ, 2005.

\bibitem{Piette:1995}
Bernard Piette and Wojtek~J Zakrzewski.
\newblock Skyrmion dynamics in (2+ 1) dimensions.
\newblock {\em Chaos, Solitons \& Fractals}, 5(12):2495--2508, 1995.

\bibitem{Sacks_Uhlenbeck}
J.~Sacks and K.~Uhlenbeck.
\newblock The existence of minimal immersions of {$2$}-spheres.
\newblock {\em Ann. of Math. (2)}, 113(1):1--24, 1981.

\bibitem{Cros}
J~Sampaio, V~Cros, S~Rohart, A~Thiaville, and A~Fert.
\newblock Nucleation, stability and current-induced motion of isolated magnetic
  skyrmions in nanostructures.
\newblock {\em Nature nanotechnology}, 2013.

\bibitem{Schoen_Uhlenbeck}
Richard Schoen and Karen Uhlenbeck.
\newblock Boundary regularity and the {D}irichlet problem for harmonic maps.
\newblock {\em J. Differential Geom.}, 18(2):253--268, 1983.

\bibitem{Schutte2014}
Christoph Sch{\"u}tte, Junichi Iwasaki, Achim Rosch, and Naoto Nagaosa.
\newblock Inertia, diffusion, and dynamics of a driven skyrmion.
\newblock {\em Physical Review B}, 90(17):174434, 2014.

\bibitem{Struwe:85}
Michael Struwe.
\newblock On the evolution of harmonic mappings of {R}iemannian surfaces.
\newblock {\em Comment. Math. Helv.}, 60(4):558--581, 1985.

\bibitem{Sulem_Sulem_Bardos:86}
P.-L. Sulem, C.~Sulem, and C.~Bardos.
\newblock On the continuous limit for a system of classical spins.
\newblock {\em Comm. Math. Phys.}, 107(3):431--454, 1986.

\bibitem{Taylor}
Michael~E. Taylor.
\newblock {\em Partial differential equations {III}. {N}onlinear equations},
  volume 117 of {\em Applied Mathematical Sciences}.
\newblock Springer, New York, second edition, 2011.

\bibitem{Wente}
Henry~C. Wente.
\newblock An existence theorem for surfaces of constant mean curvature.
\newblock {\em J. Math. Anal. Appl.}, 26:318--344, 1969.

\end{thebibliography}
\bibliographystyle{plain}

\end{document}